\documentclass{amsart}
\usepackage{shuffle}
\usepackage{amsmath}
\usepackage{amsthm}
\usepackage{amssymb}
\usepackage{ascmac}
\usepackage{color}
\usepackage[dvipdfmx]{graphicx}
\usepackage{tikz}
\usepackage{url}
\usepackage[all]{xy}

\newtheorem{defini}{Definition}[section]
\newtheorem{thm}[defini]{Theorem}
\newtheorem{prop}[defini]{Proposition}
\newtheorem{lem}[defini]{Lemma}

\newtheorem{exam}[defini]{Example}

\newtheorem{cor}[defini]{Corollary}

\newtheorem{nota}[defini]{Notation}

\numberwithin{equation}{section}

\newcommand{\N}{\mathbb{N}}
\newcommand{\Z}{\mathbb{Z}}
\newcommand{\Q}{\mathbb{Q}}
\newcommand{\R}{\mathbb{R}}
\newcommand{\C}{\mathbb{C}}

\newcommand{\re}{ {\rm{Re}} }

\newcommand{\mch}{ {\mathcal{H}} }
\newcommand{\mcz}{ {\mathcal{Z}} }

\newcommand{\rev}{ {\rm{Rev}} }
\newcommand{\tbo}{\textbf{$\omega$}}

\begin{document}
\title{ MULTIPLE ZETA FUNCTIONS AT REGULAR INTEGER POINTS}
\author{Takeshi Shinohara}
\address{Graduate School of Mathematics, Nagoya University, Furo-cho, Chikusa-ku, Nagoya,
464-8602, Japan}
\email{m21022w@math.nagoya-u.ac.jp}
\begin{abstract}
  We show the recurrence relations of the Euler-Zagier multiple zeta-function
  which describe the $r$-fold function with one variable specialized to a non-positive integer as 
  a rational linear combination of $(r-1)$-fold functions, 
  which extends the previous results of Akiyama-Egami-Tanigawa and Matsumoto. 
  As an application, we obtain an explicit method to calculate the special values of 
  the multiple zeta-function at any integer points 
  (the arguments could be neither all-positive nor all-non-positive) 
  as a rational linear summation of the multiple zeta values. 
\end{abstract}

\date{\today}
\maketitle

\tableofcontents

\setcounter{section}{-1}

\section{Introduction}
 
The Euler-Zagier {\it multiple zeta-function} (MZF for short) is the complex analytic function 
defined by the following series;
\[
  \zeta_r(s_1,s_2,\dots,s_r)
  =
  \zeta(s_1,s_2,\dots,s_r) 
   = \sum_{0<n_1<n_2<\cdots<n_r}
      \frac{1}
           {n_1^{s_1}n_2^{s_2}\cdots n_r^{s_r}},
\]
where $r$ is a positive integer and $s_1$, $s_2$, $\dots$, $s_r$ are complex variables. 
When $r=1$, it is the so-called Riemann zeta function.
It converges absolutely in the region
\[
  \re(s_r)>1,\quad \re(s_{r-1}+s_r)>2,\ \dots,\quad \re(s_1+ \cdots +s_r)>r.
\]
The special value of the MZF at positive integer points is called the {\it multiple zeta value} 
(MZV for short). The MZVs appear in calculations of a certain invariant in knot theory, 
in calculations of a certain integration in mathematical physics and in various areas of mathematics. 
See \cite{z2} for more details.
In the early 2000s, Zhao (\cite{z1}) and Akiyama-Egami-Tanigawa (\cite{AET}) 
independently showed that the MZF can be meromorphically continued to $\C^r$.
In particular, in \cite{AET}, the set of singularities of the MZF is determined as follows;
\begin{equation}\label{eqn:pole condition of mzf in introduction}
 \begin{split}
   s_r&=1, \\
   s_{r-1} + s_r &= 2, 1, 0, -2, -4,-6,\dots, \\
   \sum_{i=1}^{k}s_{r-i+1} &\in \Z_{\le k}, \quad (k=3,4,\dots,r).
 \end{split}
\end{equation}
This shows that {\it all-non-positive integer points}
\footnote{
  In this paper, {\it all-non-positive integer point} means an index $(k_1,\dots,k_r)\in\Z^r$ with 
  some $r>0$ such that $k_i\le0$ for all $i=1,\dots,r$.
         } 
mostly lie on the above singularities. 
In addition, it is known that the special values at all-non-positive integer points of the MZF 
are indeterminate and depend on a direction of the limit to choose. 
Akiyama-Egami-Tanikawa defined the special values of the MZF at all-non-positive integer points 
with a certain limit by using the following {\it recurrence relation} of the MZF 
and presented certain linear relations among those limit values.
\begin{thm}[{cf. \cite[\S2]{AET}}]
\label{original recurrence relation of MZF in introduction}
  Let $r\in\N_{\ge2}$ and $n_r\in\N_0\,(=\{0,1,2,\dots \})$.
  We have
  \begin{equation}
   \begin{split}
    \zeta_r(s_1,\dots,s_{r-1},-n_r) 
     &= -\frac{1}{n_r+1}
          \zeta_{r-1}(s_1,\dots,s_{r-2},s_{r-1}-n_r-1) \\
     &\hspace{4.5mm} 
      + \sum_{k=0}^{n_r}
         \binom{n_r}{k}
          \zeta_{r-1}(s_1,\dots,s_{r-2},s_{r-1}-n_r+k)
           \zeta(-k).
   \end{split}
  \end{equation}
\end{thm}

We note that the above formula is obtained from  the equation (6.14) in \cite{M}.
Since then, the special values of the MZF at all-non-positive integer points have been studied 
in \cite{AT}, \cite{EM}, \cite{Ka}, \cite{Ko}, \cite{MOW}, \cite{O}, \cite{S}.  

While the special values of the MZF at {\it regular integer points}, 
that is, an index $(k_1,\dots,k_r)\in\Z^r$ for some $r>0$, 
do not seem to be investigated other than the work of \cite{FKMT1}. 
Our objective in this paper is to determine the values explicitly by
extending the recurrence relations of Theorem 0.1:


\begin{thm}{\rm (Theorem \ref{thm:the extended recurrence relation of MZF}).}
\label{extended recurrence relation of MZF in introduction}
  Let $r\in\N_{\ge 2}$, 
  $j\in\{ 2,\dots,r \}$ and 
  $n_j\in\N_0$.
  We have
  \begin{equation}
    \begin{split}
      &\zeta_r(s_1,\dots,s_{j-1},-n_j,s_{j+1}\dots,s_r)
       + \zeta_{r-1}(s_1\dots,s_{j-1},s_{j+1}-n_j,s_{j+2},\dots,s_r) \\
      &\hspace{3mm}
       = -\frac{1}
               {n_j+1}
               \zeta_{r-1}(s_1,\dots,s_{j-2},s_{j-1}-n_j-1,s_{j+1},\dots,s_r) \\
      &\hspace{5.7mm}
       + \frac{1}
              {n_j+1}
              \zeta_{r-1}(s_1,\dots,s_{j-1},s_{j+1}-n_j-1,s_{j+2},\dots,s_r) \\
      &\hspace{5.7mm}
       + \sum_{k=0}^{n_j}
          \binom{n_j}{k}
           \zeta_{r-1}(s_1,\dots,s_{j-2},s_{j-1}-n_j+k,s_{j+1},\dots,s_r)\zeta(-k) \\
      &\hspace{5.7mm}
       - \sum_{k=0}^{n_j}
          \binom{n_j}{k}       
           \zeta_{r-1}(s_1,\dots,s_{j-1},s_{j+1}-n_j+k,s_{j+2},\dots,s_r)\zeta(-k). 
    \end{split}
  \end{equation}
  The following also holds when $j=1$;
  \begin{equation}
    \begin{split}
      &\zeta_r(-n_1,s_2,s_3,\dots,s_r)
        + \zeta_{r-1}(-n_1+s_2,s_3,\dots,s_r) \\
      &= \frac{1}{n_1+1}
          \zeta_{r-1}(-n_1+s_2-1,s_{3},\dots,s_r) \\
      &\hspace{3mm}
       - \sum_{k=0}^{n_1}
          \binom{n_1}{k}       
           \zeta_{r-1}(-n_1+s_2+k,s_{3},\dots,s_r)\zeta(-k) 
       + \zeta(-n_1)\zeta_{r-1}(s_2,\dots,s_r)
    \end{split}
  \end{equation}
\end{thm}

We note that when $j=r$, $(0.2)$ recovers (0.1). 
As an important application of this theorem, the following theorem is obtained.

\begin{thm}{\rm (Theorem \ref{thm:main theorem}).}
\label{main theorem in introduction}
  The special value of the MZF at any regular integer point can be presented explicitly as a 
  rational linear combination of MZVs.
\end{thm}

Explicit formulas of the special values at all regular points up to depth $3$ are presented in $\S$ 2.2.

    Saha (\cite{Sah}) obtained the Laurent expansion of $\zeta_r(s_1,\dots,s_r)$ 
    around any integer point 
    by introducing multiple Stieltjes constants $\gamma^{(a_1,\dots,a_r)}_{k_1,\dots,k_r}$, 
    where $r>0$, $(a_1,\dots,a_r)\in\Z^r$ and $k_1,\dots,k_r\in\N_0$. 
    He further observed that its regular part $\zeta_r^{\text{Reg}}(s_1,\dots,s_r)$ is expressed as 
    a finite rational linear summation of $\zeta_k(s_1,\dots,s_k)$'s ($k=1,\dots,r$).
    It should be noted that our result enables us to calculate $\gamma^{(a_1,\dots,a_r)}_{0,\dots,0}$ 
    explicitly when $(a_1,\dots,a_r)$ is a regular integer point of the MZF.

The plan of our paper goes as follows. 
In $\S$1, we review the harmonic algebra and show certain relations among words.  
In $\S$2, we translate the formulae of the words obtained in $\S$ 1 into the formulae of MZFs and prove the recurrence relations of MZFs. And then we prove the above main results and give some examples.


\section{Algebraic framework}
We prepare an algebra to deal formally with the computation of the MZF 
and show some relations on that algebra.

\subsection{The harmonic algebra}
In this subsection, we define an algebra with the structure of the harmonic product. This construction is based on \cite{fk}.

  \begin{defini}{\rm (\cite[\S1]{fk})}.\label{def:algebraic set up, non-commutative ring}
   {\rm
    Put $S:=\{s_i\}_{i\in\N}$ and  
    $S_{\N}:=\{\sum_{i=1}^ra_is_i \,|\,r\in\N,\,a_i\in\N_0,\, \sum_{i=1}^ra_i>0 \}$ 
    (i.e. $S_{\N}$ is the commutative semigroup generated by $S$).
    Let $S_{\N}^{\bullet}$ be the non-commutative free monoid generated by $S_{\N}$. 
    We denote the empty word by 1 (as the unit) and denote each elements of $S_{\N}^{\bullet}$ by 
    $(u_1, u_2,\dots,u_k) \in S_{\N}^{\bullet}$, $(u_j\in S_{\N})$ as a sequence.
    We set $\mch:=\Q\langle S_{\N}\rangle$ to be the non-commutative polynomial ring 
    generated by $S_{\N}$.
   }
  \end{defini}
Notice that $\mch$ is the $\Q$-vector space generated by $S_{\N}^{\bullet}$.
We sometimes call the element of $S_{\N}$ {\it word}.

We next review the harmonic product.
  
  \begin{defini}{\rm (\cite[\S2]{h}).}\label{def:algebraic set up, harmonic product}
   {\rm
     The {\it harmonic product} $*$ : $\mch \times \mch $ $\rightarrow$ $\mch$ is 
     the $\Q$-bilinear map defined by;
     \begin{itemize}
       \item[(i)] for any $\tbo\in S_{\N}^{\bullet}$, 
                 \[
                   \tbo * 1 = 1 * \tbo := \tbo.
                 \] 
      \item[(ii)] for any $\tbo_1$, $\tbo_2\in S_{\N}^{\bullet}$ and 
                  $u_1$, $u_2\in S_{\N}$,
                  \[
                    (\tbo_1, u_1) * (\tbo_2, u_2) 
                    := (\{\tbo_1 * (\tbo_2, u_2)\}, u_1) 
                       + (\{(\tbo_1, u_1) * \tbo_2\}, u_2) + (\{\tbo_1 * \tbo_2\}, u_1+u_2).
                  \]
     \end{itemize}
   } 
  \end{defini}

Then the pair $(\mch,*)$ forms the commutative, associative, and unital $\Q$-algebra. 
We prepare some useful notations.

\begin{nota}\label{notation:algebraic set up, length & absolute value & reverse}
 {\rm
    The map $|\bullet|\,$:$\,S_{\N}^{\bullet}\setminus\{1\}$ $\rightarrow$ $S_{\N}$ 
    is defined as follows;
    for any $\tbo=(u_1, u_2,\dots, u_k)\in S_{\N}^{\bullet}$,
    \[
      |\tbo| = u_1+u_2+\cdots+u_k\,\in S_{\N}.
    \]
    
    The map $l\,$:$\,S_{\N}^{\bullet}$ $\rightarrow$ $\N_0$ 
    is defined as follows;
    for any $\tbo=(u_1, u_2,\dots, u_k)\in S_{\N}^{\bullet}\setminus\{1\}$,
    \[
      l(\tbo) = k
    \]
    and $l(1):=0$.
    
    The map $\rev\,$:$\,S_{\N}^{\bullet}$ $\rightarrow$ $S_{\N}^{\bullet}$ 
    is defined as follows;
    for any $\tbo=(u_1, u_2,\dots, u_k)\in S_{\N}^{\bullet}\setminus\{1\}$, 
    \[
      \rev(\tbo) = (u_k, u_{k-1},\dots, u_1)\in S_{\N}^{\bullet}
    \]
    and $\rev(1):=1$.
 }
\end{nota}

\begin{defini}{\rm (\cite[\S2]{h}).}\label{def:algebraic set up, coproduct}
 {\rm
    We define the $\Q$-linear map $\Delta$ : $\mch$ $\rightarrow$ $\mch\otimes\mch$ as follows.
    \[
      \Delta = \sum_{(\alpha_1,\alpha_2)=\tbo}\alpha_1\otimes\alpha_2
      \quad (\tbo\in S_{N}^{\bullet}).
    \]
    Note that $\alpha_1$ and $\alpha_2$ can be the empty word 1. 
    We define the $\Q$-linear map $S$ : $\mch$ $\rightarrow$ $\mch$ as follows.
    \[
      S(\tbo) = (-1)^{l(\tbo)}
                 \sum_{
                  \substack{ (\alpha_1,\dots,\alpha_k)=\tbo \\
                             k\in\N\\\
                             \alpha_m \ne 1,\,(m=1,\dots,k)
                           }
                      }
                   {\rm Rev}(|\alpha_1|,\dots,|\alpha_k|)
                   \quad (\tbo \in S_{N}^{\bullet}\setminus\{1\}),
    \]
    and $S(1)=1$.
 }
\end{defini}

The triple $(*,\Delta,S)$ gives a Hopf algebra structure on $\mch$ (see \cite[Theorem 3.2]{h}).

\begin{defini}{\rm (cf. \cite[\S2, Definition 1 and 2]{im}).}
\label{def:the C-algebra generated by MZFs}
  {\rm
    Let $s_i$ be the complex parameters for $r\in\N$, $i=1,\dots,r$.
    Put $T_r:=\{t=\sum_{i=1}^ra_is_i \,|\,a_i\in\N_0,\,\sum_{i=1}^ra_i>0 \}$ and 
    $T_r^n:=\{\textbf{t}=(t_1,\dots,t_n)\,|\,t_i\in T_r \},\ (n\in\N_0)$ where 
    $T_r^0:=\emptyset$.
    We set $\mcz_r:=\Q \langle \zeta(\textbf{t})=\zeta(t_1,\dots,t_n) \,|\,
                               n\in\N_0,\,\textbf{t}\in T_r^n
                       \rangle$ 
    to be the $\Q$-linear space generated by the multiple zeta functions. We assume that 
    $\zeta(\emptyset):=1$.
  }
\end{defini}

Note that $\mcz_r$ is the subspace of meromorphic functions on $\C^r$. 
We regard $\mcz_r$ as the subset of $\mcz_{r+1}$ naturally for $r\in\N_0$. 
We look at the simplest example.

\begin{exam}\label{example:the  C-algebra generated by MZFs}
 {\rm When $r=1$, $\mcz_1$ is}
  \[
    \mcz_1=\Q\langle \zeta_n(k_1s_1,\dots,k_ns_1) \,|\,n\in\N_0,\,k_i\in\N \rangle.
  \]
\end{exam}

Put $\mcz:=\bigcup_{r=0}^{\infty}\mcz_r$ where $\mcz_0:=\Q$. 
$\mcz$ is the $\Q$-algebra because the product of two MZFs as the meromorphic functions 
can be written as the sum of the MZFs by the harmonic product.

\begin{defini}\label{def:Q-linear map from Hopf algebra to algebra generated by MZFs}
  {\rm
    We define the $\Q$-linear map 
    \[
      \zeta^*:(\mch,*) \rightarrow \mcz
    \] as follows; 
    $\zeta^*(1)=1$ and for any nonempty word $(u_1,\dots,u_k)\in S_{\N}^{\bullet}$ ($k\in\N$),
    \[
      \zeta^*(u_1,\dots,u_k):=\zeta_k(u_1,\dots,u_k).
    \]
  }
\end{defini}

Note that by construction $\zeta^*$ is a $\Q$-algebraic homomorphism from $(\mch,*)$ to $\mcz$.


\subsection{Calculation of the harmonic product}
In this subsection, we prepare some algebraic relations among words which are required to prove 
our main results in $\S$2.

\begin{lem}\label{lem:antipode relation}
  For any $\tbo\in S_{\N}^{\bullet}\setminus\{1\}$, we have the following in $(\mch, *)$;
  \[
    \sum_{
     \substack{
               (\alpha_1,\alpha_2)=\tbo
              }
         }
         (-1)^{l(\alpha_1)}f(\alpha_1)*\alpha_2
    = 0,
  \]
  where $f\,$:$\,S_{\N}^{\bullet}$ $\rightarrow$ $\mch$ is defined by $f(1)=1$ and
  \[
    f(\gamma) = \sum_{
                 \substack{
                           (\beta_1,\dots,\beta_k)=\gamma \\
                            \beta_m \ne 1
                          }
                     }
                     \rev(|\beta_1|,\dots,|\beta_k|), 
                     \quad(\gamma\in S_{\N}^{\bullet}\setminus\{1\}).
  \]
\end{lem}
%
\begin{proof}
  Let $(1\ne)\,\tbo$ be a word, $u$ : $\Q$ $\rightarrow$ $\mch$ be the unit, 
  and $\epsilon$ : $\mch$ $\rightarrow$ $\Q$ be the counit 
  (that is $\epsilon(1)=1$ and $\epsilon(\tbo)=0$ for all non-empty word $\tbo$).
  Note that the following commutative diagram holds 
  for the maps $*$, $\Delta$, $S$, $u$, and $\epsilon$;
  \[
    \xymatrix@M=8pt{
        & &\mch\otimes\mch  \ar[rr]^{S\otimes {\rm id}_{\mch}}& & \mch\otimes\mch \ar[dr]^{*}
       \ar@{}[lld]|{\circlearrowleft} \\
        &\mch\ar[ur]^{\Delta} \ar[rr]_{\epsilon}& &\Q \ar[rr]_{u} &&\mch 
    }
  \]
  Since $u\circ \epsilon(\tbo)=0$ for all non-empty word $\tbo$, we get
  \[
    * \circ (S\otimes {\rm id}_{\mch}) \circ \Delta (\tbo) = u\circ \epsilon(\tbo)=0.
  \]
  Therefore since $S(\tbo)=(-1)^{l(\tbo)}f(\tbo)$ for all word $\tbo\,(\ne1)$, we have
  \[
    * \circ (S\otimes {\rm id}_{\mch}) \circ \Delta(\tbo)
    = \sum_{
     \substack{
               (\alpha_1,\alpha_2)=\tbo
              }
         }
         (-1)^{l(\alpha_1)}f(\alpha_1)*\alpha_2
    = 0.
  \]
\end{proof}

Applying Lemma \ref{lem:antipode relation} 
to $\tbo=(u_1,u_2,\dots,u_r)$, ($u_i\in S_{\N}$) yields the following corollary.

\begin{cor}\label{cor:the special case of antipode relation}
    Let $r$ be a positive integer. For any $(u_1,u_2,\dots,u_r)\in S_{\N}^{\bullet}$ 
    ($u_i\in S_{\N}$), 
    we have
    \[
      (u_1,u_2,\dots,u_r) + \sum_{i=1}^r
                              (-1)^{i}
                            \sum_{
                              \substack{
                                        \bullet_k=,\rm{or}\, + \\
                                        2\le k \le i
                                       }
                                 }
                                 (u_i\bullet_i\cdots \bullet_2u_1)
                                 * (u_{i+1},\dots,u_r)
      = 0.
    \]
    Here the term of $i=1$ in the summation means $-(u_1)*(u_2,\dots,u_r)$.
\end{cor}

By using this corollary, we can obtain the following algebraic relation among the words.

\begin{prop}\label{prop:algebraic relation among words}
  Let $n\in\N$ and $u$, $v_1$, $v_2$, $\dots$, $v_n\in S_{\N}$
  Then we have the following equation for any word $\tbo\in S_{\N}^{\bullet}$;
  \begin{equation}
   \begin{split}
   \label{eqn:general algebraic relation among words}
    &(\tbo,u) * (v_{1},\dots,v_n) \\
    &=  (\tbo,u,v_{1},\dots,v_n) 
        + (\tbo,u+v_{1},\dots,v_n) \\
    &\hspace{3mm}
        + \sum_{i=1}^{n} 
           (-1)^{i-1} \sum_{
                       \substack{
                                 \bullet_k=,\,\rm{or}\, + \\
                                 2\le k \le i
                                }
                           }
                           (\{\tbo*(v_{i} \bullet_i \cdots \bullet_2 v_{1})\},u)
                           * (v_{i+1},\dots,v_n).
   \end{split}
  \end{equation}
  Here the term of $i=1$ in the summation is 
  \[
    (\{ \tbo * (v_1) \},u)*(v_2,\dots,v_n).
  \]
\end{prop}
%
\begin{proof}
  We show the claim by induction on $n$. 
  The claim for $n=1$ immediately follows from the harmonic product. Indeed we have
  \begin{align*}
    (\tbo,u)*(v_1) &= (\{(\tbo,u)*1\},v_1) + (\{\tbo*(v_1)\},u) + (\{\tbo*1\},u+v_1) \\
                   &= (\tbo,u,v_1) + (\tbo,u+v_1) + (\{\tbo*(v_1)\},u).
  \end{align*}
  
  Let $n\ge2$. 
  We calculate the left-hand side of the claim by using the harmonic product;
  \begin{equation}
   \begin{split}
    &(\tbo,u) * (v_{1},\dots,v_n) \\
    &= (\{(\tbo,u) * (v_{1},\dots,v_{n-1})\},v_n) + (\{\tbo * (v_{1},\dots,v_n)\},u) \\
    &\hspace{49mm}
     + (\{\tbo * (v_{1},\dots,v_{n-1})\},u + v_n).
    \intertext{
      By the induction assumption on the case of $n-1$ in the first term, we have
              }
    &= (\tbo,u,v_{1},\dots,v_{n-1},v_n) 
     + (\tbo,u+v_{1},\dots,v_{n-1},v_n) \\
    &+ \sum_{i=1}^{n-1} 
        (-1)^{i-1}
         \sum_{
          \substack{
                    \bullet_k=,\,\rm{or}\, + \\
                    2\le k \le i
                   }
              }
              (\{ {(\{\tbo*(v_{i} \bullet_i \cdots \bullet_2 v_{1})\},u)
               * (v_{i+1},\dots,v_{n-1})}\},v_n) \\
    &+ (\{\tbo * (v_{1},\dots,v_n)\},u) 
     + (\{\tbo * (v_{1},\dots,v_{n-1})\},u + v_n).
    \intertext{
      By applying the harmonic product to the summand of the third term, we have
              }
    &= (\tbo,u,v_{1},\dots,v_{n-1},v_n) 
     + (\tbo,u+v_{1},\dots,v_{n-1},v_n) \\
    &+ \sum_{i=1}^{n-1} 
        (-1)^{i-1}
         \sum_{
          \substack{
                    \bullet_k=,\,\rm{or}\, + \\
                    2\le k \le i
                   }
              }
              (\{\tbo*(v_{i} \bullet_i \cdots \bullet_2 v_{1})\},u)
              * (v_{i+1},\dots,v_{n-1},v_n) \\
    &- \sum_{i=1}^{n-1} 
        (-1)^{i-1}
         \sum_{
          \substack{
                    \bullet_k=,\,\rm{or}\, + \\
                    2\le k \le i
                   }
              }
              (\{\tbo*(v_{i} \bullet_i \cdots \bullet_2 v_{1})
              * (v_{i+1},\dots,v_{n-1},v_n)\},u)\\
    &- \sum_{i=1}^{n-1} 
        (-1)^{i-1}
         \sum_{
          \substack{
                    \bullet_k=,\,\rm{or}\, + \\
                    2\le k \le i
                   }
              }
              (\{\tbo*(v_{i} \bullet_i \cdots \bullet_2 v_{1})
              * (v_{i+1},\dots,v_{n-1})\},u+v_n)\\
    &+ (\{\tbo*(v_{1},\dots,v_n)\},u) 
     + (\{\tbo*(v_{1},\dots,v_{n-1})\},u+v_n).
    \end{split}
   \end{equation}
  
  The sum of the fourth and sixth terms 
  \begin{align*} 
    &- \sum_{i=1}^{n-1} 
        (-1)^{i-1} 
         \sum_{
          \substack{
                    \bullet_k=,\,\rm{or}\, + \\
                    2\le k \le i
                   }
              }
              (\{\tbo*(v_{i} \bullet_i \cdots \bullet_2 v_{1})
              * (v_{i+1},\dots,v_{n-1},v_n)\},u) \\
    &+ (\{\tbo*(v_{1},\dots,v_n)\},u).
  \end{align*}
  is calculated to be
  \begin{equation}
    (\{\tbo*A\},u),
  \end{equation}
  where
  \begin{align*}
    A =  (v_{1},\dots,v_r) 
         + \sum_{i=1}^{n-1} 
              (-1)^{i} \sum_{
                           \substack{
                                     \bullet_k=,\,\rm{or}\, + \\
                                     2\le k \le i
                                    }
                              }
                              (v_{i} \bullet_i \cdots \bullet_2 v_{1})
                              * (v_{i+1},\dots,v_n). 
  \end{align*}
  By using Corollary 1.9 with $u_i=v_{i}$ ($i=1,\dots,n$), we get 
  \[
      (v_{1},v_{2},\dots,v_n) 
      + \sum_{i=1}^{n}
         (-1)^{i}
          \sum_{
           \substack{
                     \bullet_k=,\rm{or}\, + \\
                     2\le k \le i
                    }
               }
               (v_{i}\bullet_i\cdots \bullet_2 v_{1})
               * (v_{i+1},\dots,v_n)
      = 0.
    \]
  Hence we have
  \begin{align*}
    &A + (-1)^{n} \sum_{
                   \substack{
                             \bullet_k=,\,\rm{or}\, + \\
                             2\le k \le n
                            }
                       }
                       (v_{n} \bullet_{n} \cdots \bullet_2 v_{1})\\
    &=(v_{1},v_{2},\dots,v_n) 
       + \sum_{i=1}^{n}
          (-1)^{i}
           \sum_{
            \substack{
                      \bullet_k=,\rm{or}\, + \\
                      2\le k \le i
                     }
                }
                (v_{i}\bullet_i\cdots \bullet_2 v_{1})
                * (v_{i+1},\dots,v_n)\\
    &= 0.
  \end{align*}
  In other words, we get 
  \begin{equation}
    A = (-1)^{n-1} \sum_{
                    \substack{
                              \bullet_k=,\,\rm{or}\, + \\
                              2\le k \le n
                             }
                        }
                        (v_{n} \bullet_{n} \cdots \bullet_2 v_{1}).
  \end{equation}
  Similarly the sum of the fifth and seventh terms
  \begin{align*}
    &- \sum_{i=1}^{n-1} 
        (-1)^{i-1}
         \sum_{
          \substack{
                    \bullet_k=,\,\rm{or}\, + \\
                    2\le k \le i
                   }
              }
              (\{\tbo*(v_{i} \bullet_i \cdots \bullet_2 v_{1})
              * (v_{i+1},\dots,v_{n-1})\},u+v_n)\\
    &+ (\{\tbo*(v_{1},\dots,v_{n-1})\},u + v_n).
  \end{align*}
  is calculated to be
  \begin{equation}
    (\{\tbo*A'\},u+v_n),
  \end{equation}
  where
  \begin{align*}
    A' =  (v_{1},\dots,v_{n-1}) 
          + \sum_{i=1}^{n-1} 
             (-1)^{i} 
              \sum_{
               \substack{
                         \bullet_k=,\,\rm{or}\, + \\
                         2\le k \le i
                        }
                   }
                   (v_{i} \bullet_i \cdots \bullet_2 v_{1})
                   * (v_{i+1},\dots,v_{n-1}). 
  \end{align*}
  By using Corollary 1.9 with $u_i=v_{i}$ ($i=1,\dots,n-1$), we get 
  \[
      (v_{1},v_{2},\dots,v_{n-1}) 
      + \sum_{i=1}^{n-1}
         (-1)^{i}
          \sum_{
           \substack{
                     \bullet_k=,\rm{or}\, + \\
                     2\le k \le i
                    }
               }
               (v_{i}\bullet_i\cdots \bullet_2 v_{1})
               * (v_{i+1},\dots,v_{n-1})
      = 0.
    \]
  In other words, we get 
  \begin{equation}
    A' = 0.
  \end{equation}
  By (1.3), (1.4), (1.5), and (1.6), we obtain the equation 
  (\ref{eqn:general algebraic relation among words}).
\end{proof}

Let $r\in\N_{\ge2}$ and $j\in \{1,2,\dots,r-1 \}$.
Applying Proposition \ref{prop:algebraic relation among words} to $n=r-j$, $u=s_j$, $v_i=s_{j+i}$ ($i=1,2,\dots,r-j$)
yields the following. 

\begin{cor}\label{cor:the special case of the algebraic relation among words}
  For $\tbo\in S_{\N}^{\bullet}$, $j\in \{1,2,\dots,r-1 \}$ we have 
  \begin{equation}
   \label{the special case of the algebraic relation among words}
   \begin{split}
    &(\tbo,s_j) * (s_{j+i},\dots,s_r) \\
    &=  (\tbo,s_j,s_{j+i},\dots,s_r) 
     + (\tbo,s_j+s_{j+i},\dots,s_r) \\
    &\hspace{3mm}
     + \sum_{i=1}^{r-j} 
        (-1)^{i-1} 
         \sum_{
          \substack{
                    \bullet_k=,\,\rm{or}\, + \\
                    2\le k \le i
                   }
              }
              (\{\tbo*(s_{j+i} \bullet_i \cdots \bullet_2 s_{j+1})\},s_j)
              * (s_{j+i+1},\dots,s_r).
   \end{split}
  \end{equation}
  Here the term of $i=1$ in the summation is 
  \[
    (\{\tbo*(s_{j+1})\},s_j)*(s_{j+2},\dots,s_r).
  \]
\end{cor}

\section{Main results}
We extend the recurrence relation of 
Theorem (\ref{original recurrence relation of MZF in introduction}) 
and give several examples of the special values of the MZFs at regular integers.

\subsection{The recurrence relation of multiple zeta functions}
In this subsection, 
we give a proof of the extended recurrence relation of the MZF 
in Theorem \ref{thm:the extended recurrence relation of MZF}.

We rewrite 
(\ref{eqn:general algebraic relation among words}) 
as the relation among MZFs and compute both sides by using the original recurrence relation 
(\ref{original recurrence relation of MZF in introduction}).

\begin{prop}\label{prop:the harmonic relation among MZFs}
  Let
  $n\in\N$, $u$, $v_1$, $v_2$, $\dots$, $v_n\in S_{\N}$, 
  $\tbo\in S_{\N}^{\bullet}$.
  we have 
  \begin{equation}
   \label{eqn:general harmonic rel of mzf}
   \begin{split}
    &\zeta(\tbo,u) \zeta(v_1,\dots,v_n) \\
    &= \zeta(\tbo,u,v_1,\dots,v_n) 
      + \zeta(\tbo,u+v_1,\dots,v_n) \\
    &\hspace{3mm}
      + \sum_{i=1}^{n} 
         (-1)^{i-1} 
          \sum_{
           \substack{
                     \bullet_k=,\,\rm{or}\, + \\
                     2\le k \le i
                    }
               }
               \zeta^*(\{\tbo*(v_i \bullet_i \cdots \bullet_2 v_1)\},u)   
                \zeta(v_{i+1},\dots,v_n).
   \end{split}
  \end{equation}
  Here the term of $i=1$ in the summation means 
  \[
    \zeta^*(\{\tbo*(v_1)\},u) \zeta(v_2,\dots,v_n).
  \]
\end{prop}
%
\begin{proof}
  We obtain the claim by mapping both sides of (1.1) using $\zeta^*$ in 
  Definition \ref{def:Q-linear map from Hopf algebra to algebra generated by MZFs}.
\end{proof}

Note that $\zeta^*(\{\tbo*(v_i \bullet_i \cdots \bullet_2 v_1)\},u)$ 
is a rational linear combination of MZFs and 
(\ref{eqn:general harmonic rel of mzf}) is the relation among meromorphic functions.
The following corollary calculates the image of the equation 
(\ref{def:Q-linear map from Hopf algebra to algebra generated by MZFs})
under the map $\zeta^*$.

\begin{cor}\label{cor:the special case of the harmonic relation among MZFs}
  For $(\tbo,u)\in S_{\N}^{\bullet}$, ($\tbo\in S_{\N}^{\bullet}$, $u\in S_{\N}$), $j\in \{1,2,\dots,r-1 \}$ we have 
  \begin{equation}
   \begin{split}
   \label{eqn:harmonic rel of mzf}
    &\zeta(\tbo,u,s_j) \zeta(s_{j+i},\dots,s_r) \\
    &=  \zeta(\tbo,u,s_j,s_{j+i},\dots,s_r) 
     + \zeta(\tbo,u,s_j+s_{j+i},\dots,s_r) \\
    &\hspace{3mm}
     + \sum_{i=1}^{r-j} 
        (-1)^{i-1} 
         \sum_{
          \substack{
                    \bullet_k=,\,\rm{or}\, + \\
                    2\le k \le i
                   }
              }
              \zeta^*(\{(\tbo,u)*(s_{j+i} \bullet_i \cdots \bullet_2 s_{j+1})\},s_j)
              * \zeta(s_{j+i+1},\dots,s_r).
   \end{split}
  \end{equation}
  Here the term of $i=1$ in the summation is 
  \[
    \zeta(\{(\tbo,u)*(s_{j+1})\},s_j)*\zeta(s_{j+2},\dots,s_r).
  \]
\end{cor}
%
\begin{proof}
  It is obtained as a special case of 
  Proposition \ref{prop:the harmonic relation among MZFs}
  with $\tbo=(\tbo,u)$, $u=s_{j}$, $n=r-j$ and $v_i=s_{j+i}$ ($i=1,\dots,r-j$).
\end{proof}

We prepare some notations 
which will be employed in a proof of the extended recurrence relation of the MZF.

\begin{nota}\label{notation:Q & Q',the product of two MZFs}
{\rm 
  Let $n\in\N$, $\tbo$, $(v_1,v_2,\dots,v_n)\in S_{\N}^{\bullet}$, $u\in S_{\N}$.
  We set 
  \begin{align*}
    Q(\tbo,u;v_{1},\dots,v_n) 
       :&=  \zeta(\tbo,u) 
            \zeta(v_1,v_2,\dots,v_n) \\
        &= \zeta(\tbo,u,v_1,\dots,v_n)
           + \zeta(\tbo,u+v_1,\dots,v_n) 
           + Q'(\tbo,u;v_{1},\dots,v_n)
  \end{align*}
  with
  \begin{align*}
   Q'(\tbo,u;v_{1},\dots,v_n)
       :&= \sum_{i=1}^{n} 
            (-1)^{i-1}
             \sum_{
              \substack{
                        \bullet_k=,\,\rm{or}\, + \\
                        2\le k \le i
                       }
                  }
                  \zeta^*(\{\tbo*(v_{i} \bullet_i \cdots \bullet_2 v_{1})\},u)   
                   \zeta(v_{i+1},\dots,v_n),
  \end{align*}
  where we put $Q(1;\tbo)=Q(\tbo;1)=\zeta(\tbo)$ and $Q'(1;\tbo)=Q'(\tbo;1)=0$.
  We note that the second equality holds by (\ref{eqn:harmonic rel of mzf}).
}
  \end{nota}
  The equation (\ref{eqn:harmonic rel of mzf}) is nothing but 
  \begin{equation}
   \label{modified harmonic relation among MZF}
   \begin{split}
    &\zeta(\tbo,u,s_j) \zeta(s_{j+i},\dots,s_r) \\
    &=  \zeta(\tbo,u,s_j,s_{j+i},\dots,s_r) 
     + \zeta(\tbo,u,s_j+s_{j+i},\dots,s_r) \\
    &\hspace{3mm}
     + Q'(\tbo,u,s_j;s_{j+i},\dots,s_r).
   \end{split}
  \end{equation}

Our strategy is to transform this equation (\ref{modified harmonic relation among MZF}) 
to an equivalent form (\ref{eqn:the extended recurrence relation of MZF with integral terms}), 
from which we will deduce the extended recurrence relation 
(Theorem \ref{thm:the extended recurrence relation of MZF}).

To start with, we review the original recurrence relation of the MZF.

\begin{thm}{\rm (cf. \cite[\S4, (4.4) and $\S$6]{M}).}\label{thm:the original recurrence relation of MZF}
  Let $\epsilon$ be a small positive real number and $M$, $n\in\N$.
  We have
  \begin{equation}
   \begin{split}
   \label{eqn:original rec rel}
    &\zeta_n(s_1,\dots,s_n) \\
    &= \frac{1}{s_n-1}\zeta_{n-1}(s_1,\dots,s_{n-1}+s_n-1)
       + \sum_{k=0}^{M-1}
          \binom{-s_n}{k}
           \zeta_{n-1}(s_1,\dots,s_{n-1}+s_n+k)\zeta(-k) \\
    &+ \frac{1}
            {2\pi i}
            \int_{(M-\epsilon)}
             \frac{\Gamma(s_n+z)\Gamma(-z)}
                  {\Gamma(s_n)}
                  \zeta_{n-1}(s_1,\dots,s_{n-1}+s_n+z)\zeta(-z)dz.
   \end{split}
  \end{equation}
  The symbol $(M-\epsilon)$ stands for the path of integration along the vertical line from 
  $M-\epsilon - i\infty$ to $M-\epsilon + i\infty$.
\end{thm}

  Let $\epsilon$ be a positive real number, $M\in\N$ and $v\in S_{\N}$. 
  We set 
  \begin{align*}
    R_{M-\epsilon}(\tbo,u,v) 
      &:= \frac{1}
              {2\pi i}
              \int_{(M-\epsilon)}
              \frac{\Gamma(v+z)\Gamma(-z)}
                   {\Gamma(v)}
                   \zeta(\tbo,u+v+z)\zeta(-z)dz.
  \end{align*}
  Then (\ref{eqn:original rec rel}) is represented as
  \begin{equation}
   \begin{split}
   \label{modified version of original rec rel with integral term}
    \zeta_n(s_1,\dots,s_n) 
    &= \frac{1}{s_n-1}\zeta_{n-1}(s_1,\dots,s_{n-1}+s_n-1) \\
    &\hspace{3mm}
     + \sum_{k=0}^{M-1}
         \binom{-s_n}{k}
          \zeta_{n-1}(s_1,\dots,s_{n-1}+s_n+k)\zeta(-k) \\
    &\hspace{3mm}+ R_{M-\epsilon}(s_1,\dots,s_{n-1},s_n).
   \end{split}
  \end{equation}

The following lemma is the calculation of the left-hand side of 
(\ref{modified harmonic relation among MZF}).

\begin{lem}\label{lem:the calculation of the LHS of the harmonic relation of MZFs}
  Let $r\in\N_{\ge 2}$, $j\in\{2,\dots,r-1\}$, 
  $\tbo\in S_{\N}^{\bullet}$ and 
  $u\in S_{\N}$.
  The following holds;
  \begin{equation}
   \begin{split}
    \label{calculation of the LHS of harmonic rel of mzf}
     \zeta(\tbo,u,s_j) \zeta(s_{j+1},\dots,s_r)
      &= \frac{1}{s_j-1} Q(\tbo,u+s_j-1;s_{j+1},\dots,s_r) \\
      &\hspace{5mm}
       + \sum_{k=0}^{M-1}
          \binom{-s_j}{k}
           \zeta(-k)Q(\tbo,u+s_j+k;s_{j+1},\dots,s_r) \\
      &\hspace{5mm}
       + \, R_{M-\epsilon}(\tbo,u,s_j)\zeta(s_{j+1},\dots,s_r). 
   \end{split}
  \end{equation}
\end{lem}
%
\begin{proof}
  By applying (\ref{eqn:original rec rel}) to $\zeta(\tbo,u,s_j)$ in the left-hand side of 
  (\ref{calculation of the LHS of harmonic rel of mzf}), 
  we obtain
  \begin{align*}
    &\zeta(\tbo,u,s_j) \zeta(s_{j+1},\dots,s_r) \\
    &= \frac{1}{s_j-1}\zeta(\tbo,u+s_j-1)\zeta(s_{j+1},\dots,s_r) \\
    &\hspace{5mm} 
     + \sum_{k=0}^{M-1}
        \binom{-s_j}{k}
         \zeta(\tbo,u+s_j+k)\zeta(s_{j+1},\dots,s_r)\zeta(-k) \\
    &\hspace{5mm} 
     +\, R_{M-\epsilon}(\tbo,u,s_j)\zeta(s_{j+1},\dots,s_r). 
    \intertext{
      Notation \ref{notation:Q & Q',the product of two MZFs} allows us to rewrite it as follows;
              }
    &= \frac{1}{s_j-1} Q(\tbo,u+s_j-1;s_{j+1},\dots,s_r) \\
    &\hspace{5mm} 
     + \sum_{k=0}^{M-1}
        \binom{-s_j}{k}
         \zeta(-k)Q(\tbo,u+s_j+k;s_{j+1},\dots,s_r) \\
    &\hspace{5mm}
     + \, R_{M-\epsilon}(\tbo,u,s_j)\zeta(s_{j+1},\dots,s_r). 
  \end{align*}
  Hence, we obtain the claim.
\end{proof}

To simplify calculation of the right-hand side of (\ref{modified harmonic relation among MZF}), 
we prepare Lemma \ref{lem:dominodaoshi for MZF} and
Notation \ref{notation:S_i for the calculation of RHS of the harmonic relation of MZFs} below.

\begin{lem}\label{lem:dominodaoshi for MZF}
  Let $\tbo\in S_{\N}^{\bullet}$ and
  $u\in S_{\N}$.
  For $(v_1,\dots,v_n)\in S_{\N}^{\bullet}$ and $n\in\N$, 
  we have
  \begin{equation}
   \label{sumofQ}
    \begin{split}
      -\zeta(\tbo,v_1+u,v_2,\dots,v_n)
      &+ \sum_{i=1}^{n}
          (-1)^{i-1} Q(\tbo,u+v_i+\cdots+v_1;v_{i+1},\dots,v_n) \\
      &+ \sum_{i=1}^{n-1}
          (-1)^{i}Q'(\tbo,u+v_i+\cdots+v_1;v_{i+1},\dots,v_n)=0,
    \end{split}
  \end{equation}
  where the empty summation is interpreted as 0.
\end{lem}
%
\begin{proof}
  It is an immediate consequence from
  \begin{equation}
   \label{eqn:rel between Q and Q'}
   \begin{split}
    &Q(\tbo,u+v_i+\cdots+v_1;v_{i+1},\dots,v_n) - Q'(\tbo,u+v_i+\cdots+v_1;v_{i+1},\dots,v_n) \\
    &= \zeta(\tbo,u+v_i+\cdots+v_1,v_{i+1},\dots,v_n)
           + \zeta(\tbo,u+v_{i+1}+v_i+\cdots+v_1,v_{i+2},\dots,v_n)
   \end{split}
  \end{equation}
  and $Q(\tbo,u+v_n+\cdots+v_1;1)=\zeta(\tbo,u+v_n+\cdots+v_1)$.
\end{proof}

The following notations are for a preparation of calculation of the equation 
(\ref{eqn:calculation of the RHS of harmonic rel of mzf}) in 
Lemma \ref{lem:the calculation of the RHS of the harmonic relation of MZFs}.

\begin{nota}\label{notation:S_i for the calculation of RHS of the harmonic relation of MZFs}
 {\rm
  Let $r, k\in\N$, $j\in\{ 1,2,\dots,r-1 \}$, $\tbo\in S_{\N}^{\bullet}$ and $u\in S_{\N}$.
  
  \begin{align*}
    S_1(k) &= \sum_{i=1}^{r-j} (-1)^{i-1}\Big(
            \zeta(\tbo,u,s_{j+i}+\cdots+s_j+k)
             + \zeta^*(\{\tbo*(s_{j+i}+\cdots+s_{j+1})\},u+s_j+k) \\
             &\hspace{23mm}
              + \zeta(\tbo,u+s_{j+i}+\cdots+s_j+k)
                 \Big) \zeta(s_{j+i+1},\dots,s_r), \\
    S_2(k) &= \sum_{i=2}^{r-j} 
                (-1)^{i-1} 
            \sum_{l=2}^i
             \sum_{
              \substack{
                        \bullet_t=,\,\rm{or}\, + \\
                        l+1\le t \le i
                       }
                  }\zeta(s_{j+i+1},\dots,s_r)\\
               &\hspace{5.5mm} \times \Bigg(
                \zeta^*
                   (\{(\tbo,u)*(s_{j+i} \bullet_i \cdots \bullet_{l+1} s_{j+l})\},
                      s_{j+l-1}+\cdots+s_{j+1}+s_j+k) \\
                &\hspace{13mm}
                 +\zeta^*
                   (\{\tbo*(s_{j+i} \bullet_i \cdots \bullet_{l+1} s_{j+l},
                      s_{j+l-1}+\cdots+s_{j+1})\},u+s_j+k) \\
                &\hspace{13mm}
                 +\zeta^*
                   (\{\tbo*(s_{j+i} \bullet_i \cdots \bullet_{l+1} s_{j+l})\},
                      u+s_{j+l-1}+\cdots+s_{j+1}+s_j+k)
                \Bigg), \\
    S_3 &= \sum_{i=1}^{r-j}(-1)^{i-1}
            \Big(
               R_{M-\epsilon}(\tbo,u,s_{j+i}+\cdots+s_{j+1},s_j) \\
             &\hspace{23mm}
              + R_{M-\epsilon}(\{\tbo*(s_{j+i}+\cdots+s_{j+1})\},u,s_j) \\
             &\hspace{23mm}
              + R_{M-\epsilon}(\tbo,u+s_{j+i}+\cdots+s_{j+1},s_j)
              \Big)\zeta(s_{j+i+1},\dots,s_r), \\
    S_4 &= \sum_{i=2}^{r-j} 
                (-1)^{i-1} 
             \sum_{l=2}^i
              \sum_{
               \substack{
                         \bullet_k=,\,\rm{or}\, + \\
                         l+1\le k \le i
                        }
                   }
                   \zeta(s_{j+i+1},\dots,s_r) \\
                   &\hspace{5.5mm} \times 
                    \Big(R_{M-\epsilon}
                          (\{(\tbo,u)*(s_{j+i} \bullet_i \cdots \bullet_{l+1} s_{j+l})\},
                            s_{j+l-1}+\cdots+s_{j+1},s_j) \\
                   &\hspace{12mm}
                    + R_{M-\epsilon}
                       (\{\tbo*(s_{j+i} \bullet_i \cdots \bullet_{l+1} s_{j+l},
                           s_{j+l-1}+\cdots+s_{j+1})\},u,s_j) \\
                   &\hspace{12mm}
                    + R_{M-\epsilon}
                       (\{\tbo*(s_{j+i} \bullet_i \cdots \bullet_{l+1} s_{j+l})\},
                           u+s_{j+l-1}+\cdots+s_{j+1},s_j)
                           \Big) .
    \end{align*}
 }
\end{nota}

The following is on the last term of right-hand side of 
(\ref{modified harmonic relation among MZF}).

\begin{lem}\label{lem:the calculation of the RHS of the harmonic relation of MZFs}
  Let $r\in\N$, $j\in\{ 2,\dots,r-1 \}$, 
  $\tbo\in S_{\N}^{\bullet}$ and 
  $u\in S_{\N}$.
  The following holds;
  \begin{equation}
   \label{eqn:calculation of the RHS of harmonic rel of mzf}
    \begin{split}
      Q'(\tbo,u,s_j;s_{j+1},\dots,s_r) 
      &= \frac{1}{s_j-1}
          \zeta(\tbo,u,s_{j+1}+s_j-1,s_{j+2},\dots,s_r) \\
      &\hspace{5mm}
       + \frac{1}{s_j-1}
          \zeta(\tbo,u+s_{j+1}+s_{j}-1,s_{j+2},\dots,s_r) \\
      &\hspace{5mm}
       + \frac{1}{s_j-1}
          Q'(\tbo,u+s_j-1;s_{j+1},\dots,s_r) \\
      &\hspace{5mm}
       + \sum_{k=0}^{M-1}
          \binom{-s_j}{k}\zeta(-k)
           \Big( S_1(k)+S_2(k) \Big) 
       + S_3 + S_4.
    \end{split}
  \end{equation}
\end{lem}
%
\begin{proof}
  To calculate $Q'$, we use the following set-theoritical decomposition; 
  for $i=2,\dots,r-j$, we have
  \begin{equation}
   \label{set-theoritical decomposition of set of bullets}
   \begin{split}
    &\{ (\bullet_i,\dots,\bullet_2) 
         \,|\, 
        \bullet_k\in\{ +\,,\,\textbf{\huge,} \}\ k=2,\dots,i \} \\
    &= \bigcup_{l=2}^{i} 
        \{ (\bullet_i,\dots,\bullet_{l+1},\,\textbf{\huge,}\,,\underbrace{+,\dots,+}_{\text{$l-2$}})
           \,|\, \bullet_k\in\{ +\,,\,\textbf{\huge,} \}\ k=l+1,\dots,i \} 
       \cup \{ (\underbrace{+,\dots,+}_{\text{$i-1$}}) \}.
   \end{split}
  \end{equation}
  Here the term of $l=2$ in the above decomposition is 
  $ (\bullet_i,\dots,\bullet_{3},\,\textbf{\huge,}) $,
  $ \bullet_k \in\{ +\,,\,\textbf{\huge,}\}$ for $k=3,\dots,i$.
  Thanks to the decomposition, we have
  \begin{equation*}
   \begin{split}
    &Q'(\tbo,u,s_j;s_{j+1},\dots,s_r) \\
    &:= \sum_{i=1}^{r-j} 
           (-1)^{i-1}
            \sum_{
             \substack{
                       \bullet_k=,\,\rm{or}\, + \\
                       2\le k \le i
                      }
                 }
                 \zeta^*(\{(\tbo,u)*(s_{j+i} \bullet_i \cdots \bullet_2 s_{j+1})\},s_j)   
                  \zeta(s_{j+i+1},\dots,s_r) \\
    &= \sum_{i=1}^{r-j}
        (-1)^{i-1}\zeta^{*}(\{(\tbo,u)*(s_{j+i} + \cdots +s_{j+1})\},s_j)
                   \zeta(s_{j+i+1},\dots,s_r) \\
    &\hspace{2mm}
     + \sum_{i=2}^{r-j} 
        (-1)^{i-1}
        \sum_{l=2}^i
         \sum_{
          \substack{
                    \bullet_k=,\,\rm{or}\, + \\
                    l+1\le k \le i
                   }
              }\\
              &\hspace{5mm}
              \zeta^*(\{(\tbo,u)*(s_{j+i} \bullet_i \cdots \bullet_{l+1} s_{j+l},
                        s_{j+l-1}+\cdots+s_{j+1})\},s_j)   
               \zeta(s_{j+i+1},\dots,s_r). 
   \end{split}
  \end{equation*}  
      By using the harmonic product, we have
  \begin{equation*}
   \begin{split}
    &Q'(\tbo,u,s_j;s_{j+1},\dots,s_r) \\
    &= \sum_{i=1}^{r-j}
        (-1)^{i-1}\zeta(\tbo,u,s_{j+i} + \cdots +s_{j+1},s_j)
                   \zeta(s_{j+i+1},\dots,s_r) \\
    &\hspace{2mm}
     + \sum_{i=1}^{r-j}
        (-1)^{i-1}\zeta^{*}(\{\tbo*(s_{j+i} + \cdots +s_{j+1})\},u,s_j)
                   \zeta(s_{j+i+1},\dots,s_r) \\
    &\hspace{2mm}
     + \sum_{i=1}^{r-j}
        (-1)^{i-1}\zeta(\tbo,u+s_{j+i} + \cdots +s_{j+1},s_j)
                   \zeta(s_{j+i+1},\dots,s_r) \\
    &\hspace{2mm} 
     + \sum_{i=2}^{r-j} 
        (-1)^{i-1}
         \sum_{l=2}^i
          \sum_{
           \substack{
                     \bullet_k=,\,\rm{or}\, + \\
                     l+1\le k \le i
                    }
               }
               \Big(A_1 + A_2 +A_3\Big)\zeta(s_{j+i+1},\dots,s_r).
   \end{split}
  \end{equation*}
      In the above equation, we set 
       \begin{align*}
         A_1 &:= \zeta^*(\{(\tbo,u)*(s_{j+i} \bullet_i \cdots \bullet_{l+1} s_{j+l})\},
                           s_{j+l-1}+\cdots+s_{j+1},s_j), \\
         A_2 &:= \zeta^*(\{\tbo*(s_{j+i} \bullet_i \cdots \bullet_{l+1} s_{j+l},
                           s_{j+l-1}+\cdots+s_{j+1})\},u,s_j), \\
         A_3 &:= \zeta^*(\{\tbo*(s_{j+i} \bullet_i \cdots \bullet_{l+1} s_{j+l})\},
                         u+s_{j+l-1}+\cdots+s_{j+1},s_j).
       \end{align*}
      By applying (\ref{modified version of original rec rel with integral term}) 
      to each term in the above equation, we obtain
  \begin{equation}
   \label{the calculation of Q' last step}
   \begin{split}
    &Q'(\tbo,u,s_j;s_{j+1},\dots,s_r) \\
    &= \frac{1}{s_j-1} 
        \sum_{i=1}^{r-j}
         (-1)^{i-1}\zeta(\tbo,u,s_{j+i} + \cdots +s_{j+1}+s_j-1)
                    \zeta(s_{j+i+1},\dots,s_r) \\
    &\hspace{2mm}
     + \frac{1}{s_j-1}
        \sum_{i=1}^{r-j}
         (-1)^{i-1}\zeta^{*}(\{\tbo*(s_{j+i} + \cdots +s_{j+1})\},u+s_j-1)
                    \zeta(s_{j+i+1},\dots,s_r) \\
    &\hspace{2mm}
     + \frac{1}{s_j-1}
        \sum_{i=1}^{r-j}
         (-1)^{i-1}\zeta(\tbo,u+s_{j+i} + \cdots +s_{j+1}+s_j-1)
                    \zeta(s_{j+i+1},\dots,s_r) \\
    &\hspace{2mm}
     + \frac{1}{s_j-1} 
        \sum_{i=2}^{r-j} 
         (-1)^{i-1} 
          \sum_{l=2}^i
           \sum_{
            \substack{
                      \bullet_k=,\,\rm{or}\, + \\
                      l+1\le k \le i
                     }
                }
                A'_1\zeta(s_{j+i+1},\dots,s_r) \\
    &\hspace{2mm}
     + \frac{1}{s_j-1} 
        \sum_{i=2}^{r-j} 
         (-1)^{i-1} 
          \sum_{l=2}^i
           \sum_{
            \substack{
                      \bullet_k=,\,\rm{or}\, + \\
                      l+1\le k \le i
                     }
                }
                A'_2 \zeta(s_{j+i+1},\dots,s_r) \\
    &\hspace{2mm}
     + \frac{1}{s_j-1} 
        \sum_{i=2}^{r-j} 
         (-1)^{i-1} 
          \sum_{l=2}^i
           \sum_{
            \substack{
                      \bullet_k=,\,\rm{or}\, + \\
                      l+1\le k \le i
                     }
                }  
                A'_3\zeta(s_{j+i+1},\dots,s_r) \\
    &\hspace{2mm}
     + \sum_{k=0}^{M-1}
          \binom{-s_j}{k}\zeta(-k)
           \Big( S_1(k)+S_2(k) \Big) 
     + S_3 + S_4
   \end{split}
  \end{equation}
  with $S_1$, $S_2$, $S_3$, $S_4$ in 
  Notation \ref{notation:S_i for the calculation of RHS of the harmonic relation of MZFs} and
  \begin{align*}
    A'_1 &:= \zeta^*(\{(\tbo,u)*(s_{j+i} \bullet_i \cdots \bullet_{l+1} s_{j+l})\},
                     s_{j+l-1}+\cdots+s_{j+1}+s_j-1), \\
    A'_2 &:= \zeta^*(\{\tbo*(s_{j+i} \bullet_i \cdots \bullet_{l+1} s_{j+l},
                     s_{j+l-1}+\cdots+s_{j+1})\},u+s_j-1), \\
    A'_3 &:= \zeta^*(\{\tbo*(s_{j+i} \bullet_i \cdots \bullet_{l+1} s_{j+l})\},
                     u+s_{j+l-1}+\cdots+s_{j+1}+s_j-1).
  \end{align*}
  We note that the each summand of the first and third terms in the right-hand side of above equation 
  (\ref{the calculation of Q' last step}) are the product of 
  two multiple zeta functions given by 
  $(-1)^{i-1}Q(\tbo,u,s_{j+i}+\cdots+s_{j}-1;s_{j+i+1},\dots,s_r)$ and 
  $(-1)^{i-1}Q(\tbo,u+s_{j+i}+\cdots+s_j-1;s_{j+i+1},\dots,s_r)$ respectively. 
  We also note that the summation of the summand of the second and fifth terms is equal to 
  $Q'(\tbo,u+s_j-1;s_{j+1},\dots,s_r)$ by definition. 
  By putting $m=i+1-l$ in the fourth and the sixth terms, we obtain
  \begin{equation*}
   \begin{split}
    &Q'(\tbo,u,s_j;s_{j+1},\dots,s_r) \\
    &= \frac{1}{s_j-1} 
        \sum_{i=1}^{r-j}
         (-1)^{i-1}Q(\tbo,u,s_{j+i}+\cdots+s_{j}-1;s_{j+i+1},\dots,s_r) \\
    &\hspace{2mm}
     + \frac{1}{s_j-1}
        Q'(\tbo,u+s_j-1;s_{j+1},\dots,s_r) \\
    &\hspace{2mm}
     + \frac{1}{s_j-1}
        \sum_{i=1}^{r-j}
         (-1)^{i-1}Q(\tbo,u+s_{j+i}+\cdots+s_j-1;s_{j+i+1},\dots,s_r) \\
    &\hspace{2mm}
     + \frac{1}{s_j-1} 
        \sum_{l=2}^{r-j} 
         (-1)^{l-1} 
          \sum_{m=1}^{r-j-l+1}
           (-1)^{m-1}
            \sum_{
               \substack{
                         \bullet_k=,\,\rm{or}\, + \\
                         l+1\le k \le m+l-1
                        }
                  }
                  A'_1\zeta(s_{j+m+l},\dots,s_r) \\
    &\hspace{2mm}
     + \frac{1}{s_j-1} 
        \sum_{l=2}^{r-j} 
         (-1)^{l-1} 
          \sum_{m=1}^{r-j-l+1}
           (-1)^{m-1}
            \sum_{
               \substack{
                         \bullet_k=,\,\rm{or}\, + \\
                         l+1\le k \le m+l-1
                        }
                  }
                  A'_3\zeta(s_{j+m+l},\dots,s_r) \\
    &\hspace{2mm}
     + \sum_{k=0}^{M-1}
          \binom{-s_j}{k}\zeta(-k)
           \Big( S_1(k)+S_2(k) \Big) 
     + S_3 + S_4.
   \end{split}
  \end{equation*}
  Since the summand of the above equation with respect to each $l$ is described in terms of $Q'$,
  we obtain
  \begin{equation*}
   \begin{split}
    &Q'(\tbo,u,s_j;s_{j+1},\dots,s_r) \\
    &= \frac{1}{s_j-1} 
        \sum_{i=1}^{r-j}
         (-1)^{i-1}Q(\tbo,u,s_{j+i}+\cdots+s_{j}-1;s_{j+i+1},\dots,s_r) \\
    &\hspace{2mm}
     + \frac{1}{s_j-1}
        Q'(\tbo,u+s_j-1;s_{j+1},\dots,s_r) \\
    &\hspace{2mm}
     + \frac{1}{s_j-1}
        \sum_{i=1}^{r-j}
         (-1)^{i-1}Q(\tbo,u+s_{j+i}+\cdots+s_j-1;s_{j+i+1},\dots,s_r) \\
    &\hspace{2mm}
     + \frac{1}{s_j-1} 
        \sum_{l=2}^{r-j} 
         (-1)^{l-1} 
          Q'(\tbo,u,s_{j+l-1}+\cdots+s_{j}-1;s_{j+l},\dots,s_r) \\
    &\hspace{2mm}
     + \frac{1}{s_j-1} 
        \sum_{l=2}^{r-j} 
         (-1)^{l-1} 
          Q'(\tbo,u+s_{j+l-1}+\cdots+s_{j}-1;s_{j+l},\dots,s_r) \\
    &\hspace{2mm}
     + \sum_{k=0}^{M-1}
          \binom{-s_j}{k}\zeta(-k)
           \Big( S_1(k)+S_2(k) \Big) 
     + S_3 + S_4.
   \end{split}
  \end{equation*}
  
      Since Lemma \ref{lem:dominodaoshi for MZF} 
      says that the summation of the first and fourth terms in the above equation is equal to 
      $\frac{1}{s_j-1}\zeta(\tbo,u,s_{j+1}+s_j-1,s_{j+2},\dots,s_r)$ 
      and the summation of the third and fifth terms is equal to
      $\frac{1}{s_j-1}\zeta(\tbo,u+s_{j+1}+s_{j}-1,s_{j+2},\dots,s_r)$, we get the equation
      (\ref{eqn:calculation of the RHS of harmonic rel of mzf}).
\end{proof}

By using Lemma \ref{lem:the calculation of the LHS of the harmonic relation of MZFs}
and Lemma \ref{lem:the calculation of the RHS of the harmonic relation of MZFs},
one can reformulate (\ref{modified harmonic relation among MZF}) as follows. 

\begin{lem}\label{lem:cal for the cal of both sides of harmonic rel of MZF}
  Let $r\in\N_{\ge 2}$, 
  $j\in\{ 2,\dots,r-1 \}$, 
  $\tbo\in S_{\N}^{\bullet}$ and 
  $u\in S_{\N}$.
  We have
  \begin{equation}
  \label{eqn:cal for the cal of RHS of harmonic rel of MZF}
   \begin{split}
     &\zeta(\tbo,u,s_j,s_{j+1},\dots,s_r)
      + \zeta(\tbo,u,s_j+s_{j+1},s_{j+2},\dots,s_r) \\
     &= \frac{1}{s_j-1}
         \zeta(\tbo,u+s_j-1,s_{j+1},\dots,s_r) 
         - \frac{1}{s_j-1}
            \zeta(\tbo,u,s_{j+1}+s_j-1,s_{j+2},\dots,s_r) \\
     &\hspace{3mm}
      + \sum_{k=0}^{M-1}
         \binom{-s_j}{k}
         \zeta(-k)Q(\tbo,u+s_j+k;s_{j+1},\dots,s_r) \\
     &\hspace{3mm}
      - \sum_{k=0}^{M-1}
         \binom{-s_j}{k}\zeta(-k)
          \Big(S_1(k)+S_2(k)\Big) 
      + R_{M-\epsilon}(\tbo,u,s_j)\zeta(s_{j+1},\dots,s_r) - S_3 - S_4. 
   \end{split}
  \end{equation}
\end{lem}
%
\begin{proof}
  Applying 
  Lemma \ref{lem:the calculation of the LHS of the harmonic relation of MZFs} 
  to the left-hand side of 
  (\ref{modified harmonic relation among MZF})
  and Lemma \ref{lem:the calculation of the RHS of the harmonic relation of MZFs} 
  to the right-hand side of
  (\ref{modified harmonic relation among MZF}), 
  we obtain
  \begin{equation}
   \label{after calculation of both sides of harmonic rel of mzf}
   \begin{split}
     &\frac{1}{s_j-1} Q(\tbo,u+s_j-1;s_{j+1},\dots,s_r) 
      + \sum_{k=0}^{M-1}
         \binom{-s_j}{k}
          \zeta(-k)Q(\tbo,u+s_j+k;s_{j+1},\dots,s_r) \\
     &\hspace{2mm}
      + \, R_{M-\epsilon}(\tbo,u,s_j)\zeta(s_{j+1},\dots,s_r) \\
     &= \zeta(\tbo,u,s_j,s_{j+1},\dots,s_r)
        + \zeta(\tbo,u,s_j+s_{j+1},s_{j+2},\dots,s_r) \\
     &\hspace{2mm} 
      + \frac{1}{s_j-1}
         \zeta(\tbo,u,s_{j+1}+s_j-1,s_{j+2},\dots,s_r)
      + \frac{1}{s_j-1}
         \zeta(\tbo,u+s_{j+1}+s_{j}-1,s_{j+2},\dots,s_r) \\
     &\hspace{2mm}
       + \frac{1}{s_j-1}
          Q'(\tbo,u+s_j-1;s_{j+1},\dots,s_r) \\
     &\hspace{2mm}
       + \sum_{k=0}^{M-1}
          \binom{-s_j}{k}\zeta(-k)
           \Big( S_1(k)+S_2(k) \Big) 
       + S_3 + S_4.
   \end{split}
  \end{equation}
 By the equation (\ref{eqn:rel between Q and Q'}), we have the identity
  \begin{equation*}
   \begin{split}
    &Q(\tbo,u+s_j-1;s_{j+1},\dots,s_r) - Q'(\tbo,u+s_j-1;s_{j+1},\dots,s_r) \\ 
    &= \zeta(\tbo,u+s_j-1,s_{j+1},\dots,s_r) + \zeta(\tbo,u+s_{j+1}+s_j-1,s_{j+2},\dots,s_r).
   \end{split}
  \end{equation*}
  By this identity, we obtain the claim.
\end{proof}

The following calculates the third and fourth terms of right-hand side of 
(\ref{eqn:cal for the cal of RHS of harmonic rel of MZF})

\begin{lem}\label{lem:almost reminder term of the cal of RHS of harmonic rel among mzf}
  Let $r$, $M\in\N$, $j\in\{ 2,\dots,r-1 \}$, 
  $\tbo\in S_{\N}^{\bullet}$ and 
  $u\in S_{\N}$.
  The following holds;  
  \begin{equation}
   \label{eqn:almost reminder term of the cal of RHS of harmonic rel among mzf}
   \begin{split}
    &\sum_{k=0}^{M-1}
      \binom{-s_j}{k}
       \zeta(-k)
        \Big( Q(\tbo,u+s_j+k;s_{j+1},\dots,s_r) - S_1(k) - S_2(k) \Big) \\
    &= \sum_{k=0}^{M-1}
        \binom{-s_j}{k}
         \zeta(\tbo,u+s_j+k,s_{j+1},\dots,s_r)\zeta(-k) \\
    &\hspace{3mm}
     - \sum_{k=0}^{M-1}
        \binom{-s_j}{k}       
         \zeta(\tbo,u,s_{j+1}+s_j+k,s_{j+2},\dots,s_r)\zeta(-k).
   \end{split}
  \end{equation}
\end{lem}

\begin{proof}
  The proof goes similarly to that of
  Lemma \ref{lem:the calculation of the RHS of the harmonic relation of MZFs}.
  We first calculate $\sum_{k=0}^{M-1}\binom{-s_j}{k}\zeta(-k)\Big(S_1(k)+S_2(k)\Big)$.
  To save the space, we set
  \[
    c_k:=\binom{-s_j}{k}\zeta(-k).
  \]
  By the definition of $S_1(k)$ and of $S_2(k)$, we have
  \begin{equation*}
   \begin{split}
    &\sum_{k=0}^{M-1}c_k
       \Big(S_1(k)+S_2(k)\Big) \\
    &= \sum_{k=0}^{M-1}c_k
        \sum_{i=1}^{r-j}(-1)^{i-1}
         \zeta(\tbo,u,s_{j+i}+\cdots+s_j+k)
          \zeta(s_{j+i+1},\dots,s_r) \\
       &\hspace{2mm}
       + \sum_{k=0}^{M-1}c_k
          \sum_{i=1}^{r-j}(-1)^{i-1}
           \zeta^*(\{\tbo*s_{j+i}+\cdots+s_{j+1}\},u+s_j+k)
            \zeta(s_{j+i+1},\dots,s_r) \\
       &\hspace{2mm}
       + \sum_{k=0}^{M-1}c_k
          \sum_{i=1}^{r-j}(-1)^{i-1}
           \zeta(\tbo,u+s_{j+i}+\cdots+s_j+k)
            \zeta(s_{j+i+1},\dots,s_r) \\
    &\hspace{2mm}
      + \sum_{k=0}^{M-1}c_k
         \sum_{i=2}^{r-j} 
          (-1)^{i-1} 
           \sum_{l=2}^i
            \sum_{
             \substack{
                       \bullet_k=,\,\rm{or}\, + \\
                       l+1\le k \le i
                      }
                 }
                 A''_1 \zeta(s_{j+i+1},\dots,s_r) \\
    &\hspace{2mm}
      + \sum_{k=0}^{M-1}c_k
         \sum_{i=2}^{r-j} 
          (-1)^{i-1} 
           \sum_{l=2}^i
            \sum_{
             \substack{
                       \bullet_k=,\,\rm{or}\, + \\
                       l+1\le k \le i
                      }
                 }
                 A''_2 \zeta(s_{j+i+1},\dots,s_r) \\
    &\hspace{2mm}
      + \sum_{k=0}^{M-1}c_k
         \sum_{i=2}^{r-j} 
          (-1)^{i-1} 
           \sum_{l=2}^i
            \sum_{
             \substack{
                       \bullet_k=,\,\rm{or}\, + \\
                       l+1\le k \le i
                      }
                 }
                 A''_3 \zeta(s_{j+i+1},\dots,s_r) \\
   \end{split}
  \end{equation*}
  In the above equation, we set
  \begin{align*}
    A''_1 &:= \zeta^*(\{(\tbo,u)*(s_{j+i} \bullet_i \cdots \bullet_{l+1} s_{j+l})\},
                       s_{j+l-1}+\cdots+s_{j+1}+s_j+k), \\
    A''_2 &:= \zeta^*(\{\tbo*(s_{j+i} \bullet_i \cdots \bullet_{l+1} s_{j+l},
                       s_{j+l-1}+\cdots+s_{j+1})\},u+s_j+k), \\
    A''_3 &:= \zeta^*(\{\tbo*(s_{j+i} \bullet_i \cdots \bullet_{l+1} s_{j+l})\},
                       u+s_{j+l-1}+\cdots+s_{j+1}+s_j+k).
  \end{align*}
  We note that the each summand of the first and third terms in the above equation are the product of 
  two multiple zeta functions given by 
  $(-1)^{i-1}Q(\tbo,u,s_{j+i}+\cdots+s_{j}+k;s_{j+i+1},\dots,s_r)$ and 
  $(-1)^{i-1}Q(\tbo,u+s_{j+i}+\cdots+s_j+k;s_{j+i+1},\dots,s_r)$ respectively. 
  We also note that the summation of the summand of the second and fifth terms is equal to 
  $Q'(\tbo,u+s_j+k;s_{j+1},\dots,s_r)$ by definition. 
  By putting $m=i+1-l$ in the fourth and the sixth terms, we obtain  
  \begin{equation*}
   \begin{split}
     &\sum_{k=0}^{M-1}c_k
       \Big(S_1(k)+S_2(k)\Big) \\
     &= \sum_{k=0}^{M-1}c_k
         \sum_{i=1}^{r-j}
          (-1)^{i-1}Q(\tbo,u,s_{j+i}+\cdots+s_{j}+k;s_{j+i+1},\dots,s_r) \\
     &\hspace{2mm}
      + \sum_{k=0}^{M-1}c_k
         Q'(\tbo,u+s_j+k;s_{j+1},\dots,s_r) \\
     &\hspace{2mm}
      + \sum_{k=0}^{M-1}c_k
        \sum_{i=1}^{r-j}
         (-1)^{i-1}Q(\tbo,u+s_{j+i}+\cdots+s_j+k;s_{j+i+1},\dots,s_r) \\
     &\hspace{2mm}
      + \sum_{k=0}^{M-1}c_k
         \sum_{l=2}^{r-j} 
          (-1)^{l-1} 
           \sum_{m=1}^{r-j-l+1}
            (-1)^{m-1}
             \sum_{
               \substack{
                         \bullet_k=,\,\rm{or}\, + \\
                         l+1\le k \le m+l-1
                        }
                  }
                  A''_1\zeta(s_{j+m+l},\dots,s_r) \\
     &\hspace{2mm}
      + \sum_{k=0}^{M-1}c_k
         \sum_{l=2}^{r-j} 
          (-1)^{l-1} 
           \sum_{m=1}^{r-j-l+1}
            (-1)^{m-1}
             \sum_{
               \substack{
                         \bullet_k=,\,\rm{or}\, + \\
                         l+1\le k \le m+l-1
                        }
                  }
                  A''_3\zeta(s_{j+m+l},\dots,s_r).
   \end{split}
  \end{equation*}
  Since the summand of the above equation with respect to each $l$ is described in terms of $Q'$,
  we obtain
  \begin{equation*}
   \begin{split}
     &\sum_{k=0}^{M-1}c_k
       \Big(S_1(k)+S_2(k)\Big) \\
     &= \sum_{k=0}^{M-1}c_k
         \sum_{i=1}^{r-j}
          (-1)^{i-1}Q(\tbo,u,s_{j+i}+\cdots+s_{j}+k;s_{j+i+1},\dots,s_r) \\
     &\hspace{2mm}
      + \sum_{k=0}^{M-1}c_k
         Q'(\tbo,u+s_j+k;s_{j+1},\dots,s_r) \\
     &\hspace{2mm}
      + \sum_{k=0}^{M-1}c_k
        \sum_{i=1}^{r-j}
         (-1)^{i-1}Q(\tbo,u+s_{j+i}+\cdots+s_j+k;s_{j+i+1},\dots,s_r) \\
     &\hspace{2mm}
      + \sum_{k=0}^{M-1}c_k
         \sum_{l=2}^{r-j} 
          (-1)^{l-1} 
           Q'(\tbo,u,s_{j+l-1}+\cdots+s_{j}+k;s_{j+l},\dots,s_r)
              \\
     &\hspace{2mm}
      + \sum_{k=0}^{M-1}c_k
         \sum_{l=2}^{r-j} 
          (-1)^{l-1} 
           Q'(\tbo,u+s_{j+l-1}+\cdots+s_{j}+k;s_{j+l},\dots,s_r).
   \end{split}
  \end{equation*}
  Since Lemma \ref{lem:dominodaoshi for MZF} 
  says that the summation of the first and fourth terms in the right-hand side above equation 
  is equal to 
  $\sum_{k=0}^{M-1}c_k\zeta(\tbo,u,s_{j+1}+s_j+k,s_{j+2},\dots,s_r)$ 
  and the summation of the third and fifth terms is equal to
  $\sum_{k=0}^{M-1}c_k\zeta(\tbo,u+s_{j+1}+s_{j}+k,s_{j+2},\dots,s_r)$, we get
  \begin{equation}
   \label{eqn:cal of sum-part of Lemma 2.8}
   \begin{split}
     &\sum_{k=0}^{M-1}c_k
       \Big(S_1(k)+S_2(k)\Big) \\
     &= \sum_{k=0}^{M-1}c_k
         \zeta(\tbo,u,s_{j+1}+s_j+k,s_{j+2},\dots,s_r) \\
     &\hspace{2mm}
      + \sum_{k=0}^{M-1}c_k
         Q'(\tbo,u+s_j+k;s_{j+1},\dots,s_r) \\
     &\hspace{2mm}
      + \sum_{k=0}^{M-1}c_k
         \zeta(\tbo,u+s_{j+1}+s_{j}+k,s_{j+2},\dots,s_r).
   \end{split}
  \end{equation}
  By (\ref{eqn:cal of sum-part of Lemma 2.8}), we have
  \begin{equation*}
   \begin{split}
    &\sum_{k=0}^{M-1}c_k
      \Big( Q(\tbo,u+s_j+k;s_{j+1},\dots,s_r) - S_1(k) - S_2(k) \Big) \\
    &= \sum_{k=0}^{M-1}c_k
        \Big( Q(\tbo,u+s_j+k;s_{j+1},\dots,s_r) - Q'(\tbo,u+s_j+k;s_{j+1},\dots,s_r) \Big)\\
    &\hspace{2mm}
      - \sum_{k=0}^{M-1}c_k
         \zeta(\tbo,u,s_{j+1}+s_j+k,s_{j+2},\dots,s_r) \\
    &\hspace{2mm}
      - \sum_{k=0}^{M-1}c_k
         \zeta(\tbo,u+s_{j+1}+s_{j}+k,s_{j+2},\dots,s_r).
   \end{split}
  \end{equation*}
  Hence, by (\ref{eqn:rel between Q and Q'}), 
  we obtain (\ref{eqn:almost reminder term of the cal of RHS of harmonic rel among mzf}).
\end{proof}

Again one can reformulate (\ref{modified harmonic relation among MZF}) 
(whence (\ref{eqn:cal for the cal of RHS of harmonic rel of MZF})) as follows.

\begin{prop}\label{prop:the extended recurrence relation of MZF with integral terms}
  Let $r\in\N_{\ge 2}$, 
  $j\in\{ 2,\dots,r-1 \}$, 
  $\tbo\in S_{\N}^{\bullet}$ and 
  $u\in S_{\N}$.
  We have
  \begin{equation}
   \label{eqn:the extended recurrence relation of MZF with integral terms}
   \begin{split}
    &\zeta(\tbo,u,s_j,\dots,s_r) 
     + \zeta(\tbo,u,s_{j+1}+s_j,\dots,s_r) \\
    &= \frac{1}
            {s_j-1}
            \left\{
                   \zeta(\tbo,u+s_j-1,\dots,s_r)   
                   - \zeta(\tbo,u,s_{j+1}+s_j-1,\dots,s_r) 
            \right\} \\
    &\hspace{3mm}
     + \sum_{k=0}^{M-1}
        \binom{-s_j}{k}
         \left\{
                \zeta(\tbo,u+s_j+k,\dots,s_r)
                 - \zeta(\tbo,s_{j+1}+s_j+k,\dots,s_r) 
         \right\}\zeta(-k) \\
    &\hspace{3mm}
     + R_{M-\epsilon}(\tbo,u,s_j)\zeta(s_{j+1},\dots,s_r) - S_3 - S_4.
   \end{split}
  \end{equation}
\end{prop}

\begin{proof}
  It is a direct consequence of 
  Lemma \ref{lem:cal for the cal of both sides of harmonic rel of MZF} and 
  Lemma \ref{lem:almost reminder term of the cal of RHS of harmonic rel among mzf}.
\end{proof}

The extended recurrence relation of the MZF is obtained by specializing 
the equation (\ref{eqn:the extended recurrence relation of MZF with integral terms}).

\begin{thm}\label{thm:the extended recurrence relation of MZF}
  Let $r\in\N_{\ge 2}$, 
  $j\in\{ 2,\dots,r-1 \}$ and 
  $n_j\in\N_0$.
  We have
  \begin{equation}
   \label{rec2}
    \begin{split}
      &\zeta_r(s_1,\dots,s_{j-1},-n_j,s_{j+1}\dots,s_r)
       + \zeta_{r-1}(s_1\dots,s_{j-1},s_{j+1}-n_j,s_{j+2},\dots,s_r) \\
      &\hspace{3mm}
       = -\frac{1}
               {n_j+1}
               \zeta_{r-1}(s_1,\dots,s_{j-2},s_{j-1}-n_j-1,s_{j+1},\dots,s_r) \\
      &\hspace{5.7mm}
       + \frac{1}
              {n_j+1}
              \zeta_{r-1}(s_1,\dots,s_{j-1},s_{j+1}-n_j-1,s_{j+2},\dots,s_r) \\
      &\hspace{5.7mm}
       + \sum_{k=0}^{n_j}
          \binom{n_j}{k}
           \zeta_{r-1}(s_1,\dots,s_{j-2},s_{j-1}-n_j+k,s_{j+1},\dots,s_r)\zeta(-k) \\
      &\hspace{5.7mm}
       - \sum_{k=0}^{n_j}
          \binom{n_j}{k}       
           \zeta_{r-1}(s_1,\dots,s_{j-1},s_{j+1}-n_j+k,s_{j+2},\dots,s_r)\zeta(-k). 
    \end{split}
  \end{equation}
\end{thm}

\begin{proof}
  It is obtained as a special case of 
  Proposition \ref{prop:the extended recurrence relation of MZF with integral terms}
  with $\tbo=(s_1,\dots,s_{j-2})$, $u=s_{j-1}$, $M-1=n_j$ and $s_j=-n_j$,
  where we set $\tbo=1$ (the empty word) when $j=2$. 
  We note that the integral terms appearing in $S_3$ and $S_4$ are $0$ because 
  \begin{align*}
    R_{M-\epsilon}(s_1,\dots,s_j) 
      &= \frac{1}
             {2\pi i}
             \int_{(M-\epsilon)}
             \frac{\Gamma(s_j+z)\Gamma(-z)}
                  {\Gamma(s_j)}
                  \zeta(s_1,\dots,s_{j-1}+s_j+z)\zeta(-z)dz
  \end{align*}
  turns to be $0$ when $s_j=-n_j$ by $\frac{1}{\Gamma(-n_j)}=0$.  
\end{proof}

The case $j=1$ is treated below.

\begin{prop}\label{prop:the extended recurrence relation of MZF for j=1}
  Let $r\in\N_{\ge 2}$ and 
  $n_1\in\N_0$.
  We have
  \begin{equation}
   \label{eqn:the extended recurrence relation of MZF for j=1}
    \begin{split}
      &\zeta_r(-n_1,s_2,s_3,\dots,s_r)
        + \zeta_{r-1}(-n_1+s_2,s_3,\dots,s_r) \\
      &= \frac{1}{n_1+1}
          \zeta_{r-1}(-n_1+s_2-1,s_{3},\dots,s_r) \\
      &\hspace{2mm}
       - \sum_{k=0}^{n_1}
          \binom{n_1}{k}       
           \zeta_{r-1}(-n_1+s_2+k,s_{3},\dots,s_r)\zeta(-k) \\
      &\hspace{2mm}
       + \zeta(-n_1)\zeta_{r-1}(s_2,\dots,s_r).
    \end{split}
  \end{equation}
\end{prop}

\begin{proof}
  By (\ref{set-theoritical decomposition of set of bullets}), we have
  \begin{equation*}
   \begin{split}
    &Q'(s_1;s_2,\dots,s_r) \\
    &:= \sum_{i=1}^{r-1} 
        (-1)^{i-1}
         \sum_{
          \substack{
                    \bullet_k=,\,\rm{or}\, + \\
                    2\le k \le i
                   }
              }
              \zeta(s_{i+1} \bullet_i \cdots \bullet_2 s_2,s_1)   
               \zeta(s_{i+2},\dots,s_r) \\
    &= \sum_{i=1}^{r-1} 
        (-1)^{i-1}
         \zeta(s_{i+1} + \cdots + s_2,s_1)   
          \zeta(s_{i+2},\dots,s_r) \\
    &\hspace{5mm}
     + \sum_{i=2}^{r-1} 
        (-1)^{i-1}
         \sum_{l=2}^i
          \sum_{
           \substack{
                     \bullet_k=,\,\rm{or}\, + \\
                     l+1 \le k \le i
                    }
               }
               \zeta(s_{i+1} \bullet_i \cdots \bullet_{l+1} s_{l+1},s_{l}+\cdots+s_2,s_1)   
                \zeta(s_{i+2},\dots,s_r).
   \end{split}
  \end{equation*}
  By (\ref{eqn:original rec rel}), we have
  \begin{equation*}
   \begin{split}
     &Q'(s_1;s_2,\dots,s_r) \\
     &=  \frac{1}{s_1-1}
          \sum_{i=1}^{r-1} 
           (-1)^{i-1}
            \zeta(s_{i+1} + \cdots + s_2+s_1-1)   
             \zeta(s_{i+2},\dots,s_r) \\
     &\hspace{5mm}
      + \sum_{k=0}^{M-1}
         \binom{-s_1}{k}
          \zeta(-k)
           \sum_{i=1}^{r-1} 
            (-1)^{i-1}
             \zeta(s_{i+1} + \cdots + s_2+s_1+k)   
              \zeta(s_{i+2},\dots,s_r) \\
     &\hspace{5mm}
      + \frac{1}{s_1-1}
         \sum_{i=2}^{r-1} 
          (-1)^{i-1} 
           \sum_{l=2}^i
            \sum_{
             \substack{
                       \bullet_k=,\,\rm{or}\, + \\
                       l+1 \le k \le i
                      }
                 }
                 A'''_1
                  \zeta(s_{i+2},\dots,s_r) \\
     &\hspace{5mm}
      + \sum_{k=0}^{M-1}
         \binom{-s_1}{k}
          \zeta(-k)
           \sum_{i=2}^{r-1} 
            (-1)^{i-1} 
             \sum_{l=2}^i
              \sum_{
               \substack{
                         \bullet_k=,\,\rm{or}\, + \\
                         l \le k \le i
                        }
                   }
                   A'''_2
                    \zeta(s_{i+2},\dots,s_r) + S_5,
   \end{split}
  \end{equation*}
  where 
  \begin{equation*}
   \begin{split}
    A'''_1&:= \zeta(s_{i+1} \bullet_i \cdots \bullet_{l+1} s_{l+1},s_{l}+\cdots+s_2+s_1-1), \\
    A'''_2&:= \zeta(s_{i+1} \bullet_i \cdots \bullet_{l+1} s_{l+1},s_{l}+\cdots+s_2+s_1+k), \\
    S_5 &:= \sum_{i=1}^{r-1}
          (-1)^{i-1}
           R_{M-\epsilon}(s_{i+1}+\cdots+s_{2},s_1)
            \zeta(s_{i+2},\dots,s_{r}) \\
      &\hspace{2mm}
       + \sum_{i=2}^{r-j} 
          (-1)^{i-1}
           \sum_{l=2}^i
            \sum_{
             \substack{
                       \bullet_k=,\,\rm{or}\, + \\
                       l+1\le k \le i
                      }
                 }
                 R_{M-\epsilon}
                  (s_{i+1} \bullet_i \cdots \bullet_{l+1} s_{l+1},s_{l}+\cdots+s_{2},s_1)   
                  \zeta(s_{i+2},\dots,s_r).
   \end{split}
  \end{equation*}
  By the definition of $Q$ in Notation \ref{notation:Q & Q',the product of two MZFs} and 
  by putting $m=i+1-l$ in the third and fourth terms, we have
  \begin{equation*}
   \begin{split}
     &Q'(s_1;s_2,\dots,s_r) \\
     &= \frac{1}{s_1-1}
         \sum_{i=1}^{r-1} 
          (-1)^{i-1}
           Q(s_{i+1} + \cdots + s_2+s_1-1;s_{i+2},\dots,s_r) \\
     &\hspace{5mm}
      + \sum_{k=0}^{M-1}
         \binom{-s_1}{k}
          \zeta(-k)
           \sum_{i=1}^{r-1} 
            (-1)^{i-1}
             Q(s_{i+1} + \cdots + s_2+s_1+k;s_{i+2},\dots,s_r) \\
     &\hspace{5mm}
      + \frac{1}{s_1-1}
         \sum_{l=2}^{r-1} 
          (-1)^{l-1}
           \sum_{m=1}^{r-l}
            (-1)^{m-1}
             \sum_{
              \substack{
                        \bullet_k=,\,\rm{or}\, + \\
                        l+1 \le k \le m+l-1
                       }
                  }
                  A'''_1
                   \zeta(s_{m+l+1},\dots,s_r)  \\
     &\hspace{5mm}
      + \sum_{k=0}^{M-1}
         \binom{-s_1}{k}
          \zeta(-k)
           \sum_{l=2}^{r-1} 
            (-1)^{l-1}
             \sum_{m=1}^{r-l}
              (-1)^{m-1}
               \sum_{
                \substack{
                          \bullet_k=,\,\rm{or}\, + \\
                          l+1 \le k \le m+l-1
                         }
                    }
                    A'''_3
                     \zeta(s_{m+l+1},\dots,s_r) + S_5.    
   \end{split}
  \end{equation*}
  Since the summand of the above equation with respect to $l$ is described in terms of $Q'$, 
  we have
  \begin{equation*}
   \begin{split}
     &Q'(s_1;s_2,\dots,s_r) \\
     &= \frac{1}{s_1-1}
         \sum_{i=1}^{r-1} 
          (-1)^{i-1}
           Q(s_{i+1} + \cdots + s_2+s_1-1;s_{i+2},\dots,s_r) \\
     &\hspace{5mm}
      + \sum_{k=0}^{M-1}
         \binom{-s_1}{k}
          \zeta(-k)
           \sum_{i=1}^{r-1} 
            (-1)^{i-1}
             Q(s_{i+1} + \cdots + s_2+s_1+k;s_{i+2},\dots,s_r) \\
     &\hspace{5mm}
      + \frac{1}{s_1-1}
         \sum_{l=2}^{r-1} 
          (-1)^{l-1}
           Q'(s_{l}+\cdots+s_2+s_1-1;s_{l+1},\dots,s_r) \\
     &\hspace{5mm}
      + \sum_{k=0}^{M-1}
         \binom{-s_1}{k}
          \zeta(-k)
           \sum_{l=2}^{r-1} 
            (-1)^{l-1}
             Q'(s_{l}+\cdots+s_2+s_1+k;s_{l+1},\dots,s_r) + S_5. 
   \end{split}
  \end{equation*}
  By applying Lemma \ref{lem:dominodaoshi for MZF} 
  to the summation of the first and third terms and that of the second and fourth terms 
  in the right-hand side of the above equation, we have 
  \begin{equation*}
   \begin{split}
     &Q'(s_1;s_2,\dots,s_r) \\
     &= \frac{1}{s_1-1}
         \zeta_{r-1}(s_1+s_2-1,s_{3},\dots,s_r) \\
     &\hspace{2mm}    
      + \sum_{k=0}^{M-1}
         \binom{-s_1}{k}
          \zeta_{r-1}(s_1+s_2+k,s_{3},\dots,s_r)\zeta(-k)
      + S_5.
   \end{split}
  \end{equation*}
  By (\ref{modified harmonic relation among MZF}), we get
  \begin{equation*}
   \label{eqn:the extended recurrence relation of MZF for j=1 with integral term}
   \begin{split}
     \zeta(s_1)\zeta(s_2,\dots,s_r)
     &= \zeta(s_1,s_2,\dots,s_r) + \zeta(s_1+s_2,\dots,s_r) \\
     &\hspace{2mm}
      + \frac{1}{s_1-1}
         \zeta_{r-1}(s_1+s_2-1,s_{3},\dots,s_r) \\
     &\hspace{2mm} 
      + \sum_{k=0}^{M-1}
         \binom{-s_1}{k}
          \zeta_{r-1}(s_1+s_2+k,s_{3},\dots,s_r)\zeta(-k)
      + S_5.
   \end{split}
  \end{equation*}
  Therefore, setting $s_1=-n_1$ for $n_1\in\N_0$ and $M-1=n_1$ in 
  (\ref{eqn:the extended recurrence relation of MZF for j=1 with integral term}), 
  we have $S_5=0$ by the same arguments as the proof of 
  Theorem \ref{thm:the extended recurrence relation of MZF}
  and we obtain the equation 
  (\ref{eqn:the extended recurrence relation of MZF for j=1}).
\end{proof}  

  \begin{thm}\label{general statement for extended recurrence relation of MZF}
    Let $j\in\{ 1,\dots,r \}$ and $n_j\in\N_0$.
    We have
    \begin{equation}
     \label{eqn:the general extended rec. rel. of mzf}
     \begin{split}
      &\zeta_r(s_1,\dots,s_{j-1},-n_j,s_{j+1}\dots,s_r) \\
      &= - (1-\delta_{j1})
            \frac{1}{n_j+1}
             \zeta_{r-1}(s_1,\dots,s_{j-2},s_{j-1}-n_j-1,s_{j+1},\dots,s_r) \\
      &\hspace{2mm}
       + (1-\delta_{jr})
          \frac{1}{n_j+1}
           \zeta_{r-1}(s_1,\dots,s_{j-1},-n_j+s_{j+1}-1,s_{j+2},\dots,s_r) \\
      &\hspace{2mm}
       + (1-\delta_{j1})
          \sum_{k=0}^{n_j}
           \binom{n_j}{k}
            \zeta_{r-1}(s_1,\dots,s_{j-2},s_{j-1}-n_j+k,s_{j+1},\dots,s_r)\zeta(-k) \\
      &\hspace{2mm}
       - (1-\delta_{jr})
          \sum_{k=0}^{n_j}
           \binom{n_j}{k}       
            \zeta_{r-1}(s_1,\dots,s_{j-1},-n_j+s_{j+1}+k,s_{j+2},\dots,s_r)\zeta(-k) \\
      &\hspace{2mm}
       + \delta_{j1}(1-\delta_{jr})\zeta_j(s_1,\dots,s_{j-1},-n_j)\zeta_{r-j}(s_{j+1},\dots,s_r) \\
      &\hspace{2mm}
       + (\delta_{jr}-1) \zeta_{r-1}(s_1\dots,-n_j+s_{j+1},\dots,s_r),  
     \end{split}
    \end{equation}
    where $\delta_{ij}$ is Kronecker's delta.
  \end{thm}
%
\begin{proof}
By combining Theorem \ref{original recurrence relation of MZF in introduction},
Theorem \ref{thm:the extended recurrence relation of MZF}, and
Proposition \ref{prop:the extended recurrence relation of MZF for j=1},
we obtain the claim.
\end{proof}


\subsection{Special values of multiple zeta functions at regular integer points}
In this subsection, we explain that special values of MZF 
at any regular integer points are calculated explicitly by 
(\ref{eqn:the general extended rec. rel. of mzf}).
And we present explicit formulas of those special values up to depth 3.

To begin with, we mention that 
Theorem \ref{general statement for extended recurrence relation of MZF} 
gives the recurrence relation among special values of MZF at regular integer points.
  \begin{cor}\label{cor:a special value of mzf at regular integer point}
    Let $r\in\N$ and 
    an index $(n_1,\dots,-n_j,\dots,n_r)\in\Z^r$ ($j\in\{ 1,\dots,r \}$, $n_j\in\N_0$) 
    be not on the set of singularities of MZF given by 
    (\ref{eqn:pole condition of mzf in introduction}).
    We have
    \begin{equation}
     \label{eqn:a special value of mzf at regular integer point}
     \begin{split}
      &\zeta_r(n_1,\dots,n_{j-1},-n_j,n_{j+1}\dots,n_r) \\
      &= - (1-\delta_{j1})
            \frac{1}{n_j+1}
             \zeta_{r-1}(n_1,\dots,s_{j-2},n_{j-1}-n_j-1,n_{j+1},\dots,n_r) \\
      &\hspace{2mm}
       + (1-\delta_{jr})
          \frac{1}{n_j+1}
           \zeta_{r-1}(n_1,\dots,n_{j-1},-n_j+n_{j+1}-1,n_{j+2},\dots,n_r) \\
      &\hspace{2mm}
       + (1-\delta_{j1})
          \sum_{k=0}^{n_j}
           \binom{n_j}{k}
            \zeta_{r-1}(n_1,\dots,n_{j-2},n_{j-1}-n_j+k,n_{j+1},\dots,n_r)\zeta(-k) \\
      &\hspace{2mm}
       - (1-\delta_{jr})
          \sum_{k=0}^{n_j}
           \binom{n_j}{k}       
            \zeta_{r-1}(n_1,\dots,n_{j-1},-n_j+n_{j+1}+k,n_{j+2},\dots,n_r)\zeta(-k) \\
      &\hspace{2mm}
       + \delta_{j1}(1-\delta_{jr})\zeta_j(n_1,\dots,n_{j-1},-n_j)\zeta_{r-j}(n_{j+1},\dots,n_r) \\
      &\hspace{2mm}
       + (\delta_{jr}-1) \zeta_{r-1}(n_1\dots,-n_j+n_{j+1},\dots,n_r),  
     \end{split}
    \end{equation}
    where $\delta_{ij}$ is Kronecker's delta.
  \end{cor}
%
\begin{proof}
  Though this is a direct consequence from (\ref{eqn:the general extended rec. rel. of mzf}), 
  it is worthy to note that indices $(n_1,\dots,n_{j-1}-n_j+k,\dots,n_r)$ and 
  $(n_1,\dots,n_{j+1}-n_j+k,\dots,n_r)$ ($k\in\{-1,0,\dots,n_j \}$) 
  are the regular point of MZF in each term on the right-hand side of 
  (\ref{eqn:a special value of mzf at regular integer point}).
\end{proof}


To state our main theorem precisely, we prepare some notations.
For an index $\textbf{k}:=(k_1,\dots,k_r)\in\Z^r$ ($r\in\N$) with
\begin{align*}
  k_{i_1},\dots,k_{i_l}>0 \quad (1\le i_1<\cdots<i_l\le r,\ l\in\N), \\
  k_j\le 0                \quad (j\in\{1,\dots,r\}\setminus\{ i_1,\dots,i_l\}), 
\end{align*}
we set $\textbf{k}^{+}:=(k_{i_1},\dots,k_{i_l})$, ${\rm dp }\,\textbf{k}^{+}:=l$ and 
${\rm wt}\,\textbf{k}^{+}:=k_{i_1}+\cdots+k_{i_l}$.
We also define the following $\Q$ linear subspace of $\R$ generated by MZVs.
For $a\in\N_0$ and $b\in\N_{\ge2}$, we define
\[
  Z^{\le a}_{\le b} := \langle\, 
                        \zeta(\textbf{m}) \,|\, 
                         \textbf{m}\in\N^{n},\ n\in\{0,1,\dots,a\},\ 
                         {\rm wt}\,\textbf{m}\le b\,
                       \rangle_{\Q},
\]
where $\textbf{m}=(m_1,\dots,m_n)\in\N^{n}$ is an admissible index, i.e. $m_n>1$ 
and we set $\zeta(\emptyset):=1$.

Our main theorem is stated as follows.

\begin{thm}\label{thm:main theorem}
  The equation (\ref{eqn:a special value of mzf at regular integer point}) allows us to express 
  the special value of the MZF at a regular integer point as a 
  rational linear combination of MZVs.
  In more detail, when an index $\textbf{n}=(n_1,\dots,n_r)\in\Z^r$ with
  $d={\rm dp}\,\textbf{n}^+$ and $w={\rm wt}\,\textbf{n}^+$ is not on the set 
  of  singularities given by (\ref{eqn:pole condition of mzf in introduction}), 
  we have
  \[
    \zeta_r(n_1,\dots,n_r)\in Z^{\le d}_{\le w}.
  \]
\end{thm}

\begin{proof}
  We show the claim by induction on $r$. Let $r=2$. 
  If the variables $s_1$ and $s_2$ are both positive integers, 
  $\zeta_2(s_1,s_2)$ is a double zeta value. 
  Hence we may assume that at least one of $s_1$ and $s_2$ is non-positive. 
  By \cite[$\S$4, Lemma 4.1]{FKMT1}, we have 
  \begin{equation*}
   \begin{split}
     \zeta_2(s_1,-n_2)
        &= -\frac{1}{n_2+1}
                 \zeta(s_1-n_2-1) 
           + \sum_{k=0}^{n_2}
              \binom{n_2}{k}
               \zeta(s_1-n_2+k)\zeta(-k),  \\
     \zeta_2(-n_1,s_2)
        &= \frac{1}{n_1+1}
            \zeta(s_2-n_1-1) 
           - \sum_{k=0}^{n_1}
              \binom{n_1}{k}       
               \zeta(s_2-n_1+k)\zeta(-k) \\
        &\hspace{2.7mm}
         + \zeta(-n_1)\zeta(s_2)
           - \zeta(s_2-n_1).       
   \end{split}
  \end{equation*}
  Since for $k\in\N_0$, we have $\zeta(-k) = -\frac{B_{k+1}}{k+1}$ 
  (see Example \ref{example for r=1}).
  Hence, we obtain the claim for $r=2$.
  
  \indent
  Let $r>2$. 
  If the all variables $s_1$,$\dots$,$s_r$ are positive integers, 
  $\zeta_r(s_1,\dots,s_r)$ is nothing but a multiple zeta value. 
  Hence we may consider the special value for an index 
  $(n_1,\dots,-n_j,\dots,n_r)\in\Z^r$ ($j\in\{ 1,\dots,r \}$, $n_j\in\N_0$) 
  which are not on the set of singularities of MZF.
  By using the equation (\ref{eqn:a special value of mzf at regular integer point}) 
  and by the induction assumption,
  one can easily check that 
  $\zeta_r(n_1,\dots,-n_j,\dots,n_r)$ is in $Z^{\le d}_{\le w}$ 
  with $d={\rm dp}\,\textbf{n}^+$ and $w={\rm wt}\,\textbf{n}^+$.
  Thus we obtain the claim.
\end{proof}

In precise, the equation (\ref{eqn:a special value of mzf at regular integer point}) allows us to calculate special values of
$\zeta_r(s_1,\dots,s_r)$ at regular integer points explicitly.
The following is an example in the case of $r=1$.

\begin{exam}\label{example for r=1}
  {\rm
    It is well known that $\zeta(n)$ is the Riemann zeta value when $n\in\Z_{\ge1}$ 
    and for $n\in\Z_{\le0}$, we have
    \[
      \zeta(-n) = -\frac{B_{n+1}}{n+1},
    \]
    where $B_n$ is the Bernoulli number which is a rational number defined by
    \[
      \frac{t}{e^t-1} = \sum_{n=0}^{\infty}\frac{B_n}{n!}t^n.
    \]
  }
\end{exam}

The following is an example in the case of $r=2$.

\begin{exam}\label{example:r=2}
  {\rm 
  (cf. \cite[Lemma 4.1]{FKMT2}). 
  We assume that $(n_1,n_2)\in\Z^2$ is a regular integer point of $\zeta_2(s_1,s_2)$,
  i.e. 
  \[
    n_2\ne1\quad \text{and}\quad n_1+n_2\in\Z_{>2}\cup\{ -k \,|\, k\in\Z_{>0},\,k:\text{odd} \}.
  \]
  Then the special value $\zeta_2(s_1,s_2)$ is calculated as follows;
  \begin{enumerate}
    \item[(i)] 
    When $(n_1,n_2)\in\Z_{\ge1}\times\Z_{\ge2}$, $\zeta_2(n_1,n_2)$ is a double zeta value. 
    \item[(ii)] 
    When $(n_1,n_2)\in\Z_{\ge1}\times\Z_{\le0}$, we have
    \[
      \zeta_2(n_1,n_2) 
      = -\frac{1}{-n_2+1}
          \zeta(n_{1}+n_2-1)
        - \sum_{k=0}^{-n_2}
           \binom{-n_2}{k}
            \frac{B_{k+1}}{k+1}
             \zeta(n_1+n_2+k).
    \]
    \item[(iii)] 
    When $(n_1,n_2)\in\Z_{\le0}\times\Z_{\ge2}$, we have
    \begin{equation*}
     \begin{split}
       \zeta_2(n_1,n_2) 
       &= \frac{1}{-n_1+1}
           \zeta(n_{1}+n_2-1)
          + \sum_{k=0}^{-n_1}
             \binom{-n_1}{k}
              \frac{B_{k+1}}{k+1}
               \zeta(n_1+n_2+k) \\
        &\hspace{3mm}   
         + \zeta(n_1)\zeta(n_2)
          - \zeta(n_1+n_2).
     \end{split}
    \end{equation*}
    \item[(iv)] 
    When $(n_1,n_2)\in\Z_{\le0}\times\Z_{\le0}$ with $n_1n_2\ne 0$, we have
    \[
      \zeta_2(n_1,n_2) = -\frac{1}{2}\zeta(n_1+n_2).
    \]
    \item[(v)] 
    When $(n_1,n_2)\in\Z_{\le0}\times\Z_{\le0}$ with $n_1n_2= 0$, we have
    \[
      \zeta_2(n_1,n_2) = \zeta(n_1+n_2).
    \]
  \end{enumerate}
  } 
\end{exam}

The following is an example in the case of $r=3$.

\begin{exam}
 {\rm 
  We assume that $(n_1,n_2,n_3)\in\Z^3$ is a regular integer point of $\zeta_3(s_1,s_2,s_3)$,
  i.e. 
  \[
    n_3\ne1 \quad \text{and}\quad 
    n_2+n_3\in\Z_{>2}\cup\{ -k \,|\, k\in\Z_{>0},\,k:\text{odd} \} \quad \text{and}\quad 
    n_1+n_2+n_3\in\Z_{>3}.
  \]
  Then the special value $\zeta_3(s_1,s_2,s_3)$ is calculated as follows;
  \begin{enumerate}
    \item[(i)] 
    When $(n_1,n_2,n_3)\in\Z_{\ge1}\times\Z_{\ge1}\times\Z_{\ge2}$, 
    $\zeta_3(n_1,n_2,n_3)$ is a triple zeta value.
    \item[(ii)] 
    When $(n_1,n_2,n_3)\in\Z_{\ge1}\times\Z_{\ge1}\times\Z_{\le0}$, we have
    \begin{equation*}
     \begin{split}
       \zeta_3(n_1,n_2,n_3)
       &= - \frac{1}{-n_3+1}
             \zeta_2(n_1,n_2+n_3-1) \\
       &\hspace{2mm}
        + \sum_{k=0}^{-n_3}
           \binom{-n_3}{k}
            \frac{B_{k+1}}{k+1}
             \zeta_2(n_1,n_2+n_3+k).
     \end{split}
    \end{equation*} 
    \item[(iii)] 
    When $(n_1,n_2,n_3)\in\Z_{\ge1}\times\Z_{\le0}\times\Z_{\ge2}$, we have
    \begin{equation*}
     \begin{split}
       &\zeta_3(n_1,n_2,n_3)\\
       &= - \frac{1}{-n_2+1}
             \{
               \zeta_2(n_1+n_2-1,n_3) 
               - \zeta_2(n_1,n_2+n_3-1) 
             \} \\
       &\hspace{2mm}
        + \sum_{k=0}^{-n_2}
           \binom{-n_2}{k}
            \frac{B_{k+1}}{k+1}
             \{
               \zeta_2(n_1+n_2+k,n_3) - \zeta_2(n_1,n_2+n_3+k)
             \} \\
       &\hspace{2mm}
        - \zeta_2(n_1,n_2+n_3).
     \end{split}
    \end{equation*} 
    \item[(iv)] 
    When $(n_1,n_2,n_3)\in\Z_{\le0}\times\Z_{\ge1}\times\Z_{\ge2}$, we have
    \begin{equation*}
     \begin{split}
       \zeta_3(n_1,n_2,n_3)
       &=  \frac{1}{-n_1+1}
            \zeta_2(n_1+n_2-1,n_3) \\
       &\hspace{2mm}
        - \sum_{k=0}^{-n_1}
           \binom{-n_1}{k}
            \frac{B_{k+1}}{k+1}
             \zeta_2(n_1+n_2+k,n_3) \\
       &\hspace{2mm}
        - \frac{B_{-n_1+1}}{-n_1+1}\zeta_2(n_2,n_3) - \zeta_2(n_1+n_2,n_3).
     \end{split}
    \end{equation*} 
    \item[(v)] 
    When $(n_1,n_2,n_3)\in\Z_{\ge1}\times\Z_{\le0}\times\Z_{\le0}$, we have
    \begin{equation*}
     \begin{split}
       \zeta_3(n_1,n_2,n_3) 
       &= \frac{1}{-n_3+1}
           \frac{1}{-n_2-n_3+2}
            \zeta(n_1+n_2+n_3-2) \\
       &\hspace{2mm}
       + \frac{1}{-n_3+1} 
          \sum_{l=0}^{-n_2-n_3+1}
           \binom{-n_2-n_3+1}{l}
            \frac{B_{l+1}}{l+1}
             \zeta(n_1+n_2+n_3-1+l) \\
       &\hspace{2mm}
        + \sum_{k=0}^{-n_3}
           \binom{-n_3}{k}
            \frac{1}{-n_2-n_3-k+1}
             \frac{B_{k+1}}{k+1}
              \zeta(n_1+n_2+n_3-1+k) \\
       &\hspace{2mm}
        - \sum_{k=0}^{-n_3}
           \binom{-n_3}{k}          
            \sum_{l=0}^{-n_2-n_3-k}
             \binom{-n_2-n_3-k}{l}
              \frac{B_{k+1}}{k+1}
               \frac{B_{l+1}}{l+1}
                \zeta(n_1+n_2+n_3+k+l).
     \end{split}
    \end{equation*} 
    \item[(vi)] 
    When $(n_1,n_2,n_3)\in\Z_{\le0}\times\Z_{\ge1}\times\Z_{\le0}$, we have
    \begin{equation*}
     \begin{split}
       &\zeta_3(n_1,n_2,n_3) \\
       &= -\frac{1}{-n_3+1}
            \frac{1}{-n_1+1}
             \zeta(n_1+n_2+n_3-2) 
          + \frac{1}{-n_3+1} 
             \zeta(n_1+n_2+n_3-1) \\
       &\hspace{2mm}
        - \frac{1}{-n_3+1}
           \sum_{l=0}^{-n_1}
            \binom{-n_1}{l}
             \frac{B_{l+1}}{l+1}
              \zeta(n_1+n_2+n_3-1+l) \\
       &\hspace{2mm}
        + \sum_{k=0}^{-n_3}
           \binom{-n_3}{k}
            \frac{B_{k+1}}{k+1}
             \left\{
              -\frac{1}{-n_1+1}
                 \zeta(n_1+n_2+n_3+k-1)
                + \zeta(n_1+n_2+n_3+k) \right\} \\
       &\hspace{2mm}
        - \sum_{k=0}^{-n_3}
           \binom{n_3}{k}          
            \sum_{l=0}^{-n_1}
             \binom{-n_1}{l}
              \frac{B_{l+1}}{l+1}
               \frac{B_{k+1}}{k+1}
                \zeta(n_1+n_2+n_3+k+l).
     \end{split}
    \end{equation*} 
    \item[(vii)] 
    When $(n_1,n_2,n_3)\in\Z_{\le0}\times\Z_{\le0}\times\Z_{\ge2}$, we have
    \begin{equation*}
     \begin{split}
       &\zeta_3(n_1,n_2,n_3) \\
       &= \frac{1}{-n_1+1}
           \frac{1}{-n_1-n_2+2}
            \zeta(n_1+n_2+n_3-2) \\
       &\hspace{2mm}
          - \frac{B_{-n_1+1}}{-n_1+1}
             \frac{1}{-n_2+1} 
              \zeta(n_2+n_3-1)\\
       &\hspace{2mm}
        - \left( \frac{1}{-n_1-n_2+1} + \frac{1}{-n_1+1} \right) 
           \zeta(n_1+n_2+n_3-1)
        + \frac{B_{-n_1+1}}{-n_1+1}
           \zeta(n_2+n_3)  \\
       &\hspace{2mm}
        + \left(
            \frac{B_{-n_1-n_2+1}}{-n_1-n_2+1} 
            + \frac{B_{-n_1+1}}{-n_1+1}\frac{B_{-n_2+1}}{-n_2+1}
            - \frac{1}{-n_1+1}\frac{B_{-n_1-n_2+2}}{-n_1-n_2+2}
          \right)
           \zeta(n_3) \\
       &\hspace{2mm}     
        + \zeta(n_1+n_2+n_3) \\
       &\hspace{2mm}
        + \frac{1}{-n_1+1}
           \sum_{l=0}^{-n_1-n_2+1}
            \binom{-n_1-n_2+1}{l}
             \frac{B_{l+1}}{l+1}
              \zeta(n_1+n_2+n_3-1+l) \\
       &\hspace{2mm}
        + \sum_{k=0}^{-n_1}
           \binom{-n_1}{k}
            \frac{B_{k+1}}{k+1}
             \frac{1}{-n_1-n_2-k+1}
              \zeta(n_1+n_2+n_3-1+k) \\
       &\hspace{2mm}
        - \sum_{k=0}^{-n_1}
           \binom{-n_1}{k}          
            \sum_{l=0}^{-n_1-n_2-k}
             \binom{-n_1-n_2-k}{l}
              \frac{B_{l+1}}{l+1}
               \frac{B_{k+1}}{k+1}
                \zeta(n_1+n_2+n_3+k+l) \\
       &\hspace{2mm}
        - \sum_{k=0}^{-n_1}
           \binom{-n_1}{k}
            \frac{B_{k+1}}{k+1}
             \left( 
               \frac{B_{-n_1-n_2-k+1}}{-n_1-n_2-k+1} 
                \zeta(n_3) 
               + \zeta(n_1+n_2+n_3+k) 
             \right) \\
       &\hspace{2mm}
        - \frac{B_{-n_1+1}}{-n_1+1}
           \sum_{l=0}^{-n_2}
            \binom{-n_2}{l}
             \frac{B_{l+1}}{l+1}
              \zeta(n_2+n_3+l) \\
       &\hspace{2mm}       
        - \sum_{k=0}^{-n_1-n_2}
           \binom{-n_1-n_2}{l}
             \frac{B_{l+1}}{l+1}
              \zeta(n_1+n_2+n_3+l).
     \end{split}
    \end{equation*} 
    \item[(viii)] 
    When $(n_1,n_2,n_3)\in\Z_{\le0}\times\Z_{\le0}\times\Z_{\le0}$, 
    the index $(n_1,n_2,n_3)$ is on the set of singularities of $\zeta_3(s_1,s_2,s_3)$. 
    The special value there is indeterminate.
  \end{enumerate}  
  
 }
\end{exam}

\noindent{\it Acknowledgments.} 
The author sincerely expresses his gratitude to H. Furusho for leading him to this area and giving him many valuable comments.
He also profoundly appreciates N. Komiyama for his helpful advice.

\end{document}